\documentclass[11pt]{amsart}
\usepackage{amssymb}
\usepackage{amsfonts}
\usepackage{color}

\newtheorem{theorem}{Theorem}
\theoremstyle{plain}

\newtheorem{definition}{Definition}

\newtheorem{lemma}{Lemma}

\newtheorem{proposition}{Proposition}
\newtheorem{remark}{Remark}

\numberwithin{equation}{section}
\begin{document}
\title{\textbf{Hermitian functional representation
of free L\'{e}vy processes}}%\footnote{Part of this work was supported by CONACYT Grant A1-S-9764.}}
\author{Jose-Luis P\'{e}rez G.}
\address{Departamento de Probabilidad y Estad\'{\i}stica\\
CIMAT, Guanajuato, M\'{e}xico.\\
jluis.garmendia@cimat.mx}
\author{V\'{\i}ctor P\'{e}rez-Abreu}
\address{Instituto Polit\'{e}cnico Nacional, M\'{e}xico.\\
pabreuv@gmail.com}
\author{Alfonso Rocha-Arteaga.}
\address{Facultad de Ciencias F\'{\i}sico-Matem\'{a}ticas\\
Universidad Aut\'{o}noma de Sinaloa, M\'{e}xico.\\
arteaga@uas.edu.mx}

\begin{abstract}
A functional representation of free L\'evy
processes is established via an ensemble of unitarily
invariant Hermitian matrix-valued L\'evy processes. This is accomplished by proving functional
asymptotics of their empirical spectral processes towards the law of a free L\'evy
processes. This result recovers a functional version of Wigner's
theorem and introduces a functional version of Marchenko-Pastur's theorem
providing the free Poisson process as the noncommutative limit process.

\ \ \

\end{abstract}

\maketitle

\section{Introduction}

Let $\left \{  B^{(n)}(t)\right \}  _{t\geq0}=\{(b_{jk}(t))\}_{t\geq0}$ be the
$n\times n$ Hermitian matrix-valued Brownian motion where $(b_{jj}%
(t))_{j=1}^{n},$ $(\mathrm{Re}b_{jk}(t))_{j<k},$ $(\mathrm{Im}b_{jk}%
(t))_{j<k}$ is a set of $n^{2}$ independent one-dimensional Brownian motions
with parameter $\frac{t}{2}(1+\delta_{jk}).$ This matrix-valued process was
first considered by Dyson \cite{Dy62} and the study of its eigenvalue process
leads to several primary results in {Random Matrix Theory} (RMT), noncolliding
particles, free probability, and laws of noncommutative processes.

For any fixed $t>0$, $B^{(n)}(t)$ is a Gaussian Unitary Ensemble (GUE)
of random matrices with parameter $t$ and its matrix distribution is invariant
under unitary conjugation as well as infinitely divisible with respect to the
classical convolution of matrix distributions, \cite{A}, \cite{Me}. Let
$\{(\lambda_{1}(t),\lambda_{2}(t),...,\lambda_{n}(t))\}_{t\geq0}$ be the
$n$-dimensional stochastic process of the eigenvalues of $B^{(n)}$ and
consider the empirical spectral process of the re-scaled process %re-scaled matrix
$B^{(n)}/\sqrt{n}$
\begin{equation}
\mu_{t}^{(n)}=\frac{1}{n}\sum_{j=1}^{n}\delta_{\widetilde{\lambda}_{j}%
	(t)},\quad t\geq0,\label{22}%
\end{equation}
where $\widetilde{\lambda}_{j}(t)=\lambda_{j}(t)/\sqrt{n}$ and $\delta_{x}$ is
the unit mass at $x.$

First, from the fundamental work of Wigner \cite{Wig} in RMT, for each fixed $t>0,$
$\mu_{t}^{(n)}$ converges as $n$ goes to infinity, weakly almost surely, to
the semicircle (Wigner) distribution with parameter $t$\textcolor{green}{,}%:

\begin{equation}
\mathrm{w}_{t}(\mathrm{d}x)=\frac{1}{2\pi t}\sqrt{4t-x^{2}}{\ 1}_{\left[
	-2\sqrt{t},2\sqrt{t}\right]  }(x)\mathrm{d}x;\label{22a}%
\end{equation}
see also \cite{A}, \cite{Me}, \cite{Ta12}. The Wigner distribution
$\mathrm{w}_{t}$ is infinitely divisible with respect to the free
convolution and it also appears as the limiting distribution in the free
central limit theorem \cite{hp}, \cite{NS},\  \cite{VDN}.

In this sense, $\mathrm{w}_{1}$ is the free counterpart of the Gaussian
distribution in classical infinite divisibility, playing in free probability
the role the Gaussian distribution does in classical probability. This \ is
the starting point of the subject of free infinite divisibility, \cite{hp},
\cite{NS},\  \cite{VDN}. Moreover, the family $\left \{  \mathrm{w}_{t}\right \}
_{t\geq0}$ is the law of free Brownian motion, a family of selfadjoint
elements $\left \{  Z_{t}\right \}  _{t\geq0}$ in a noncommutative probability
space that has free increments and is such that for each $0\leq s\leq t$,
$Z_{t}-Z_{s}$ has the law $\mathrm{w}_{t-s}$; see Biane \cite{Bi}.

Second, keeping $n\geq1$ fixed, in a pioneering work, Dyson \cite{Dy62}
realized that the eigenvalue dynamics is described by a diffusion process with
non-smooth drift satisfying the It\^{o} Stochastic Differential Equation
(SDE)
\begin{equation}
\mathrm{d}\lambda_{i}(t)=\mathrm{d}W_{i}(t)+\sum_{j\neq i}\frac{\mathrm{d}%
	t}{\lambda_{i}(t)-\lambda_{j}(t)},\quad t\geq0,1\leq i\leq n,\label{21}%
\end{equation}
where $W_{1},...,W_{n}$ are independent one-dimensional standard Brownian
motions, see also \cite{A}, \cite{Ta12}. The stochastic process $\left \{
\left(  \lambda_{1}(t),\lambda_{2}(t),...,\lambda_{n}(t)\right)  \right \}
_{t\geq0}$ is called the Dyson Brownian motion corresponding to the GUE. This
is a primary example of a system of interacting particles governed by a SDE
with strong interactions due to the non-smooth drift coefficient, a phenomenon
associated with the process of eigenvalues of several matrix continuous-time
processes; see \cite{Bru1}, \cite{Bru2}.

Third, a functional version of Wigner's theorem is possible: this follows from
the study of the limiting laws of measure-valued processes of interacting
diffusions with non-smooth drift coefficient, as considered by Rogers and Shi
\cite{rosh}. Namely, the empirical measure-valued processes $\{ \{ \mu
_{t}^{(n)}\}_{t\geq0}:n\geq1\}$ of the strong interacting particles (\ref{21})
converge to the family of measures $\left \{  \mathrm{w}_{t}\right \}  _{t\geq
	0}.$ This functional asymptotic takes place in $C\left(  \mathbb{R}_{+}%
,\Pr(\mathbb{R})\right)  $, the space of continuous functions from
$\mathbb{R}_{+}\ $into$\  \Pr(\mathbb{R})$ endowed with the topology of
uniform convergence on compact intervals of $\mathbb{R}_{+}$, where
$\Pr(\mathbb{R})\ $is the space of probability measures on $\mathbb{R} $
endowed with the topology of weak convergence; see also \cite{A}, \cite{CG}.

A similar functional asymptotic behavior in the case of other matrix
(continuous-time) diffusions has been considered, by Chan \cite{cha}, and Rogers and Shi \cite{rosh} leading to free
Brownian motion; and by \cite{CG}, \cite{AT} to other noncommutative processes like
the dilation of the free Poisson process. The latter is a
functional version of the Marchenko--Pastur law, but the noncommutative
limiting process is not a free L\'{e}vy process. The case of a
symmetric fractional Brownian motion was considered in \cite{PPP}, obtaining
the non-commutative law of the fractional Brownian motion introduced in
\cite{NT}.

The goal of this paper is to establish a functional representation of free L\'evy processes
by matrix L\'evy processes. This
generalizes the results for free Brownian motion to general free L\'{e}vy
processes. While the classical works of Chang \cite{cha} and Rogers and Shi \cite{rosh} deal with continuous diffusions, our matrix processes have jumps. Up to the best of our knowledge, this is the first time that convergence of measure-valued empirical processes
arising from matrix processes with jumps are considered.

More specifically, the main result of this paper is summarized as follows. Let
$\mathcal{D}(\mathbb{R}_{+},\mathrm{Pr}(\mathbb{R}))$ denote the space of
right continuous functions with left limits from $\mathbb{R}_{+}\ $%
into$\  \Pr(\mathbb{R}),$ endowed with the Skorohod topology, where
$\Pr(\mathbb{R})\ $is the space of probability measures on $\mathbb{R}$
endowed with the topology of weak convergence. For a given $n\times n$
Hermitian process $\left \{  X_{t}^{(n)}:t\geq0\right \}  $ let, for each
$t\geq0$, $\lambda_{1}^{(n)}(t)\geq \lambda_{2}^{(n)}(t)\geq \cdots \geq
\lambda_{n}^{(n)}(t) $ denote the ordered eigenvalues of $X_{t}^{(n)}$. The
empirical spectral measure-valued process of $\left \{  X_{t}^{(n)}%
:t\geq0\right \}  $ is defined as%
\begin{equation}
\mu_{t}^{(n)}(dx)=\frac{1}{n}\sum_{m=1}^{n}\delta_{\lambda_{m}^{(n)}%
	(t)}(dx)\text{, }t\geq0.\label{mvp}%
\end{equation}

\begin{theorem}
	\label{main} Given a free L\'{e}vy process $\left \{ Z_{t}:t\geq0\right \} $, there
	exists an ensamble of matrix L\'{e}vy processes $\left \{  X_{t}^{(n)}%
	:t\geq0\right \}  _{n\geq1}$, such that:
	
	a) For each $n\geq1,$ $\left \{  X_{t}^{(n)}:t\geq0\right \}  $ is a unitarily
	invariant L\'{e}vy process in the space of $n\times n$ Hermitian matrices.
	
	b) For each $n>1$ and $t>0$, the matrix distribution of $X_{t}^{(n)}$ is
	absolutely continuous with respect to Lebesgue measure on $\mathbb{R}^{n^{2}}
	$ and $P\left(  X_{t}^{(n)}\text{ has simple spectrum}\right)  =1$.
	
	c) For each $n\geq1$, the non-zero jumps of $X_{t}^{(n)}$, $\Delta X_{t}%
	^{(n)}=X_{t}^{(n)}-X_{t-}^{(n)}$, are of rank one.
	
	d) The empirical spectral measure-valued processes $\left \{  \mu_{t}%
	^{(n)}:t\geq0\right \}  _{n\geq1}$ converge weakly in probability to the law of $\left \{ Z_{t}:t\geq0\right \} $ as $n\rightarrow \infty$, in the space
	$\mathcal{D}(\mathbb{R}_{+},\mathrm{Pr}(\mathbb{R}))$.
\end{theorem}

We should emphasize that our main contribution is the functional convergence in probability in the space $\mathcal{D}(\mathbb{R}_{+},\mathrm{Pr}(\mathbb{R}))$ of the sequence of empirical spectral processes given by \eqref{mvp} to the law of a free L\'evy process. In particular this implies the convergence of finite dimensional distributions which can be obtained by the results in \cite{BG} and \cite{ca}, together with Voiculescu's asymptotic freeness theorem found in \cite{V00} for independent unitary invariant Hermitian random matrices.

As a particular case, we obtain a functional version of the Marchenko--Pastur
theorem, where the asymptotic noncommutative process is the free Poisson
process.

The strategy to prove Theorem \ref{main} and the needed principal results
are as follows. Section 2 contains the background on Hermitian L\'{e}vy
processes and free L\'{e}vy processes. Section 3
introduces the ensembles of Hermitian L\'{e}vy processes $\left \{
X_{t}^{(n)}:t\geq0\right \} _{n\geq1}$ of Theorem \ref{main} via the
appropriate characteristic triplets associated to the free L\'{e}vy process $%
\left \{ Z_{t}:t\geq0\right \} $ and with the properties (a)--(c) in Theorem %
\ref{main}. Section 4 deals with the dynamics of the semimartingales of the
eigenvalues of $X_{t}^{(n)}$, for which we use an It\^{o} formula due
to \cite{PV} and helpful asymptotics for the associated local martingales
are also proved. Section 5 presents the proof of the tightness of the
spectral measure-valued processes $\{ \mu_{t}^{(n)}:t\geq0\}_{n\geq1}$ in
the space $\mathcal{D}(\mathbb{R}_{+},\mathrm{Pr}(\mathbb{R}))$, which, as
expected, is more complex than the Brownian case due to the jumps$.$ The
key facts are that for each $n\geq1$, all the jumps of the Hermitian L\'{e}%
vy process $\left \{ X_{t}^{(n)}:t\geq0\right \} $ are of rank one. Finally, 
Theorem \ref{thfbm} in Section 6 characterizes the Burgers
equation satisfied by the Cauchy transform of the limiting family of laws $\{ \mu_{t}:t\geq0\}$ of
the  sequence of spectral measure valued processes $\{ \mu_{t}^{(n)}:t\geq0\}_{n\geq1}$
in $\mathcal{D}(\mathbb{R}_{+},\mathrm{Pr}(\mathbb{R}))$. This allows us to identify the family  $\{
\mu_{t}:t\geq0\}$ as the law of the free L\'{e}vy process $\left \{
Z_{t}:t\geq0\right \} $, employing a result of Bercovici and Voiculescu \cite%
{BV} for free infinitely divisible measures with unbounded support (see also \cite{ca} and \cite{HS}).

\section{Preliminaries on Hermitian and free L\'{e}vy processes}

\subsection{Unitary invariant Hermitian L\'{e}vy processes}

In this section we consider a class of Hermitian L\'{e}vy processes whose
distributions are invariant under unitary conjugation.

Let $\mathbb{M}_{n}=\mathbb{M}_{n}\left( \mathbb{C}\right) $ denote the
linear space of $n\times n$ matrices with complex entries with scalar
product $\left \langle A,B\right \rangle $ $=~\mathrm{tr}\left(
B^{\ast}A\right) $ and the Frobenius norm $\left \Vert A\right \Vert =\left[
\mathrm{tr}\left( A^{\ast}A\right) \right] ^{1/2}$ where $\mathrm{tr}$
denotes the (non-normalized) trace. The set of Hermitian matrices in $%
\mathbb{M}_{n}$ is denoted by $\mathbb{H}_{n}$, $\mathbb{H}_{n}^{0}=$ $%
\mathbb{H}_{n}\backslash \{0\}$ and $\mathbb{H}_{n}^{1}$ is the set of rank
one matrices in $\mathbb{H}_{n}.$ Let $\mathbb{S}_{n}$ denote the unit
sphere in $\mathbb{H}_{n}$, let $\mathbb{S}(\mathbb{H}_{n}^{1})=\mathbb{S}%
_{n}\cap \mathbb{H}_{n}^{1}$ and let $\overline{\mathbb{H}}_{n}^{+}$ denote
the set of nonnegative definite Hermitian matrices.
We denote the $n\times n$ identity matrix by $\mathrm{I}_{n}$.

A random matrix $X$ in $\mathbb{H}_{n}$ is infinitely divisible if for all $%
m\geq1$ there exist independent identically distributed random matrices $%
X_{1},...,X_{m}$ in $\mathbb{H}_{n}$ such that $X_{1}+...+X_{m}$ and $X$
have the same matrix distribution. In this case, the matrix distribution of $%
X$ is characterized by the L\'{e}vy--Khintchine representation of its
Fourier transform $\mathbb{E}\mathrm{e}^{\mathrm{itr}(\Theta
X)}=\exp(\varphi (\Theta))$ with Laplace exponent
\begin{equation}
\varphi(\Theta)={}\mathrm{itr}(\Theta \Psi_{n}\text{ }){}-{}\frac{1}{2}%
\mathrm{tr}\left( \Theta \mathcal{A}_{n}\Theta \right) {}+{}\int _{\mathbb{H}%
_{n}}\left( \mathrm{e}^{\mathrm{itr}(\Theta \xi)}{}-1{}-\mathrm{i}\frac{%
\mathrm{tr}(\Theta \xi)}{1+\left \Vert \xi \right \Vert ^{2}}{}\right)
\nu_{n}(\mathrm{d}\xi),\ \Theta \in \mathbb{H}_{n},  \label{LKRgen}
\end{equation}
where $\mathcal{A}_{n}:\mathbb{H}_{n}\rightarrow \mathbb{H}_{n}$ is a linear
operator which is positive $(i.e.$ $\mathrm{tr}\left( \Phi \mathcal{A}%
_{n}\Phi \right) \geq0$ for $\Phi \in \mathbb{H}_{n})$ and symmetric $(i.e.$
$\mathrm{tr}\left( \Theta_{2}\mathcal{A}_{n}\Theta_{1}\right) =\mathrm{tr}%
\left( \Theta_{1}\mathcal{A}_{n}\Theta_{2}\right) $ for $\Theta_{1},%
\Theta_{2}\in \mathbb{H}_{n})$, $\nu_{n}$ is a measure on $\mathbb{H}_{n}$
(the L\'{e}vy measure) satisfying $\nu_{n}(\{0\})=0$ and $\int_{\mathbb{H}%
_{n}}(\left \Vert \xi \right \Vert ^{2}\wedge1)\nu_{n}(\mathrm{d}\xi)<\infty$%
, and $\Psi_{n}\in \mathbb{H}_{n}$. The triplet $(\mathcal{A}%
_{n},\Psi_{n},\nu_{n})$ is unique.

The following is straightforward.

\begin{proposition}
\label{polar} Fix $n\geq1$ and let $X_{n}$ be an infinitely divisible $%
n\times n$ \ Hermitian random matrix with L\'{e}vy--Khintchine
representation (\ref{LKRgen}) with L\'{e}vy triplet $(\mathcal{A}_{n},
\Psi_{n},\nu_{n})$, where

\noindent a) $\Psi_{n}=\mathcal{\gamma}\mathrm{I}_{n},\gamma \in \mathbb{R}.$
%{\color{red} [creo que si vale la pena definir $\mathrm{I}_{n}$, escribi la definicion atras]}

\noindent b) $\mathcal{A}_{n}\Theta=\frac{\sigma_{n}^{2}}{n}\Theta$, $\Theta
\in \mathbb{H}_{n}$, and $\sigma_{n}^{2}\geq0$.

\noindent c) $\nu_{n}(UEU^{\ast})=$ $\nu_{n}(E)$ for each unitary $n\times n$
nonrandom matrix $U$ and $E\in \mathcal{B}\left( \mathbb{H}_{n}^{0}\right) .$

Then the distribution of $X_{n}$ is invariant under unitary conjugation.
\end{proposition}

\begin{definition}
An $n\times n$ matrix-valued process $\left \{ X(t) :t\geq0\right \} $ is a
Hermitian L\'{e}vy process if for each $t>0,$ $X(t)\in \mathbb{H}_{n}$ and

\noindent i) $X(0)=0$ with probability one,

\noindent ii) $X$ has independent increments: $\forall$ $0\leq t_{1}<\cdot
\cdot \cdot<t_{m},m\geq1,$ $X(t_{m})-X(t_{m-1}),...,X(t_{2})-X(t_{1})$ are
independent random matrices,

\noindent iii) $X$ has stationary increments: $\forall$ $0\leq s<t,$ $%
X(t)-X(s)$ and $X(t-s)$ have the same matrix distribution, and

\noindent iv) for any $s\geq0$, the increment $X(t+s)-X(s)\rightarrow
\mathbf{0}_{n}$ in distribution as $t\rightarrow0$, where $\mathbf{0}_{n}$
is the $n\times n$ zero matrix.
\end{definition}

A key feature of an $n\times n$ Hermitian L\'{e}vy process $X(t)$ with
triplet given by Proposition \ref{polar} is that for each $t>0$, the
distribution of $X(t)$ is invariant under unitary conjugation. Furthermore,
the nonzero jumps $\Delta X(t)=X(t)-X(t-)$ are random matrices of rank one.

Given any infinitely divisible $n\times n$ Hermitian random matrix $X$,
there is a Hermitian L\'{e}vy process $\left \{ X(t):t\geq0\right \} $ such
that $X$ and $X(1)$ have the same distribution, and vice versa. In fact, $%
X(t)$ has the Fourier transform $\mathbb{E}[e^{\text{\textrm{itr}}(\Theta
X(t))}]=\exp(t\varphi(\Theta))$, where $\varphi$ is the above Laplace
exponent.

Throughout this paper we will assume that $X(t)$ has an absolutely
continuous distribution for each $t>0$. In order for this condition to hold,
we will ask that $X$ have a Gaussian component $\left(
\sigma^{2}\neq0\right) $ or that it satisfies condition $\mathbf{D}$ in \cite%
{PV}. Under this assumption, for each $t>0$, $X(t)$ has a simple spectrum
\cite{PV}.

The following dynamics for the eigenvalues of a class of Hermitian L\'{e}vy
processes is proved in \cite{PV}.

\begin{proposition}
\label{PVeigen} Let $\left \{ X(t):t\geq0\right \} $ be an $n\times n$
Hermitian L\'{e}vy process with absolutely continuous distribution invariant
under unitary conjugation, and with triplet $(\sigma^{2}\mathrm{I}%
_{n}\otimes \mathrm{I}_{n}$,$\gamma \mathrm{I}_{n}$,$\nu)$. Let $(\lambda
_{1}(t),...,\lambda_{n}(t))$ be the vector of eigenvalues of $X(t)$ where $%
\lambda_{1}(t)>\lambda_{2}(t)>\cdots>\lambda_{n}(t)$ for each $t\geq0$. For
each $m=1,...,n,$ the eigenvalue $\lambda_{m}$ is a semimartingale and
\begin{align}
\lambda_{m}(X_{t}) & =\lambda_{m}(X_{0})+\gamma
\sum_{i=1}^{n}\int_{0}^{t}(D\lambda_{m}(X_{s-}))_{ii}ds+\sigma^{2}%
\int_{0}^{t}\sum_{j\not =m}\frac{1}{\lambda_{m}(X_{s-})-\lambda_{j}(X_{s-})}%
ds+M_{t}^{m}  \label{LFE} \\
& +\int_{(0,t]\times \mathbb{H}_{n}^{0}}\left[ \lambda_{m}(X_{s-}+y)-%
\lambda_{m}(X_{s-})-\mathrm{tr}(D\lambda_{m}(X_{s-})y)1_{\left \{ \left \Vert
y\right \Vert \leq1\right \} }\right] \nu(dy)ds,  \notag
\end{align}
 with
\begin{equation*}
M_{t}^{m}=\sigma
\sum_{r=1}^{n}\sum_{l=1}^{n}\int_{0}^{t}(D%
\lambda_{m}(X_{s-}))_{rl}dB_{s}^{rl}+\int_{(0,t]\times \mathbb{H}%
_{n}^{0}}\left[\lambda _{m}(X_{s-}+y)-\lambda_{m}(X_{s-})\right]\widetilde{J}%
_{X}(ds,dy),
\end{equation*}
where $J_{X}(\cdot,\cdot)$ is the Poisson random measure of the jumps of $X$
on $[0,\infty)\times \mathbb{H}_{n}^{0}$ with intensity measure $Leb\otimes
\nu $, independent of a family of independent one dimensional standard
Brownian motions $B_{s}^{ij},i,j=1,...,n$ and the compensated measure is
given by
\begin{equation*}
\widetilde{J}_{X}(\mathrm{d}t,\mathrm{d}y)=J_{X}(\mathrm{d}t,\mathrm{d}y)-%
\mathrm{d}t\nu(\mathrm{d}y)\text{;}
\end{equation*}
and {for each $s\geq0$}, $D\lambda_{m}(X_{s})$ is the matrix of derivatives
of $\lambda_{m}(X_{s})$ with respect to the entries of $X_{s}$, given by{\
\begin{equation}
\left( D\lambda_{m}(X_{s})\right) _{ij}=2\overline{u}_{im}(s)u_{jm}(s)1_{%
\{i<j\}}+|u_{im}(s)|^{2}1_{\{i=j\}}\text{,}  \label{ED}
\end{equation}
} where $u_{ij}(s)$ $i,j=1,2,...,n$ are the entries of a unitary random
matrix $U_s$ of eigenvectors of $X_s$.
\end{proposition}

\begin{remark}
\label{mart rep} If we take
\begin{equation*}
\widetilde{M}_{t}^{m}:=\sum_{r=1}^{n}\sum_{l=1}^{n}\int_{0}^{t}(D\lambda
_{m}(X_{s-}))_{rl}dB_{s}^{rl},
\end{equation*}
it is clear that for $m,m^{\prime}=1,\dots,n$ its covariation process is
given by
\begin{equation*}
\langle \widetilde{M}^{m},\widetilde{M}^{m^{\prime}}\rangle_{t}=t\delta
_{mm^{\prime}}\qquad \text{$t>0$.}
\end{equation*}
Therefore, by L\'{e}vy's Theorem, we can write, for $m=1,\dots,n$, the
martingale term $M^{m}$ as
\begin{equation*}
M_{t}^{m}=\sigma W_{t}^{m}+\int_{(0,t]\times \mathbb{H}_{n}^{0}}\left[\lambda
_{m}(X_{s-}+y)-\lambda_{m}(X_{s-})\right]\widetilde{J}_{X}(ds,dy),\qquad \text{for
$t>0$,}
\end{equation*}
where $W^{1},\dots,W^{n}$ are independent one dimensional standard Brownian
motions.
\end{remark}

\subsection{Free L\'{e}vy processes affiliated with $W^{\ast}$-probability
spaces}

We recall some facts on free L\'{e}vy processes acting on a $W^{\ast}$%
-probability space. For additional information on this subject, see \cite{A}%
, \cite{BNT06}, \cite{BV}, and \cite{Bi2}. A $W^{\ast}$-probability space is a pair $(%
\mathcal{G},\tau)$ where $\mathcal{G}\ $is a von Neumann algebra acting on a
Hilbert space $H$ and $\tau$ is a normal faithful trace on $\mathcal{G}$. In
the sequel, $(\mathcal{G},\tau)$ will denote a $W^{\ast}$-probability space.

An unbounded operator $a$ in $H$ is not an element of
$\mathcal{G}$. However, a selfadjoint linear operator $a$ in $H$ is
affiliated with $\mathcal{G}$ if and only if $f(a)\in \mathcal{G}$ for any
bounded Borel function $f:\mathbb{R\rightarrow R}$. Here $f(a)$ is defined
in the sense of spectral theory (the functional calculus). That is, for any
selfadjoint operator $a$ affiliated with $\mathcal{G}$, there exists a
unique probability measure $\mu_{a}$ on $\mathbb{R}$, concentrated on the
spectrum of $a$, such that
\begin{equation*}
\tau(f(a))=\int_{\mathbb{R}}f(s)\mu_{a}(ds)\text{,}
\end{equation*}
for every bounded Borel function $f:\mathbb{R\rightarrow R}$. The measure $%
\mu_{a}$ is called the (spectral) distribution of $a$ and is denoted by $%
\mu_{a}=\mathcal{L}\left \{ a\right \} $. Unless $a$ is bounded, the
spectrum of $a$ is an unbounded subset of $\mathbb{R}$ and, in general, $%
\mu_{a}$ is not compactly supported.

\begin{definition}
\label{FIaff} Let $a_{1},a_{2},...,a_{r}$ be selfadjoint operators
affiliated with a $W^{\ast}$-probability space $(\mathcal{G},\tau)$. It is
said that $a_{1},a_{2},...,a_{r}$ are freely independent with respect to $%
\tau$ if for any bounded Borel functions $f_{1},f_{2},...,f_{r}:\mathbb{%
R\rightarrow R}$, the bounded linear operators $%
f_{1}(a_{1}),f_{2}(a_{2}),...,f_{r}(a_{r})$ in $\mathcal{G}$ are freely
independent with respect to $\tau$. That is,
\begin{equation}
\tau \left \{ \lbrack f_{_{i_{1}}}(a_{i_{1}})-\tau(f_{i_{_{1}}}(a_{i_{1}}))]
\left[ f_{i_{2}}(a_{i_{2}})-\tau \left( f_{i_{2}}(a_{i_{2}})\right) \right]
\cdots \left[ f_{_{i_{m}}}(a_{i_{m}})-\tau \left(
f_{_{i_{m}}}(a_{i_{m}})\right) \right] \right \} =0  \label{Wfreeness}
\end{equation}
for any positive integer $m$ and any $i_{1},i_{2},...,i_{m}$ in $\left \{
1,2,...,r\right \} $ with $i_{1}\neq i_{2},i_{2}\neq i_{3}...,i_{m-1}\neq
i_{m}$.
\end{definition}

{A stochastic process affiliated with a $W^{\ast}$-probability space $(%
\mathcal{G},\tau)$ is a family $\left \{ Z_{t}:t\geq0\right \} $ of
selfadjoint operators affiliated with $\mathcal{G}$}. Let us denote by $%
\mu_{t}=\mathcal{L}\left \{ Z_{t}\right \} $ the (spectral) distribution of $%
Z_{t}$ for each $t\geq0$. The family $\left \{ \mu_{t}:t\geq0\right \} $ of
probability measures on $\mathbb{R}$ is called the family of spectral
distributions of the process $\left \{ Z_{t}:t\geq0\right \} $. Moreover,
for any $s\geq0,t\geq0$ such that $s\leq t$, the increment $Z_{t}-Z_{s}$ is
again a selfadjoint operator affiliated with $\mathcal{G}$ and we denote its
distribution by $\mu_{s,t}=\mathcal{L}\left \{ Z_{t}-Z_{s}\right \} .$

\begin{definition}
\label{FLP} A free L\'{e}vy process is a stochastic process $\left \{
Z_{t}:t\geq0\right \} $ affiliated with the $W^{\ast}$-probability space $(%
\mathcal{G},\tau)$\ such that:

$i)$ $Z_{0}=0$.

$ii)$ For any $m\geq1$ and $0\leq t_{1}<\cdots<t_{m}$, the increments%
\begin{equation*}
Z_{t_{m}}-Z_{t_{m-1}},...,Z_{t_{2}}-Z_{t_{1}}
\end{equation*}
are freely independent random variables.

$iii)$ For any $s\geq0,t\geq0$ the spectral distribution of $Z_{t+s}-Z_{s}$
does not depend on $s$.

$iv)$ For any $s\geq0$, the increment $Z_{t+s}-Z_{s}\rightarrow0$ in
distribution as $t\rightarrow0$, that is, the spectral distributions $%
\mathcal{L}\left \{ Z_{t+s}-Z_{s}\right \} $ converge weakly to $\delta_{0}$
as $t\rightarrow0$.
\end{definition}

It is well known that the law $\upsilon=\mathcal{L}(Z_{1})$ of a free L\'{e}%
vy process $\left \{ Z_{t}:t\geq0\right \} $ is free infinitely divisible.
Moreover, it has the L\'{e}vy--Khintchine representation $%
\phi_{Z_{t}}(z)=t\phi_{\upsilon}(z)$ in terms of the {Voiculescu} transform
\begin{equation}
\phi_{\upsilon}(z)=\eta+\int_{\mathbb{R}}\frac{1+tz}{z-t}\rho(\mathrm{d}%
t),\quad(z\in \mathbb{C}^{+}),  \label{LKpair}
\end{equation}
with generating pair $(\eta,\rho)$, where $\eta \in \mathbb{R}$ and $\rho$
is a finite measure on $\mathbb{R}$, see \cite{BNT06}, \cite{BV}, \cite{VDN}.

{Finally, let $ID(\ast)$ and $ID(\boxplus)$ be the set of classical and
free infinitely divisible distributions on $\mathbb{R}$, respectively.
The Bercovici--Pata bijection \cite{BP}, denoted by $\Lambda$, between $ID(\ast)$ and $ID(\boxplus)$, 
is such that for each $\mu \in ID(\ast)$ with L\'{e}vy triplet $%
(\sigma^{2},\gamma ,\nu)$, $\Lambda(\mu)\in ID(\boxplus)$ has generating pair%
}
\begin{align}
\rho(\mathrm{d}x) & =\sigma^{2}\delta_{0}(\mathrm{d}x)+\frac{x^{2}}{1+x^{2}}%
\nu(\mathrm{d}x),  \label{T1} \\
\eta & =\gamma-\int_{\mathbb{R}}x\left( 1_{\left[ -1,1\right] }(x)-\frac{1}{%
1+x^{2}}\right) \nu(\mathrm{d}x),  \label{T2}
\end{align}
see \cite{BNT04}, \cite{BNT06}.

Benaych--Georges \cite{BG} and Cabanal--Duvillard
\cite{ca} explained the bijection $\Lambda$ via random matrix models. Their work
constitutes a generalization of the Wigner semicircle law for the GUE to more
general random matrices. The distributions of these random matrices share
similar properties to those of the GUE, such as having an infinitely divisible
matrix distribution which is invariant under unitary conjugation
(\textit{L\'{e}vy unitary ensemble}).

\section{The approximating Hermitian L\'{e}vy processes}\label{sectionHL}

In this section we introduce the ensemble of Hermitian 
valued processes considered in this paper. Let $\left \{ Z_{t}:t\geq0\right \} $ be
a free L\'{e}vy process with generating pair $(\eta,\rho)$ and let $%
\sigma^{2}=\rho(\left \{ 0\right \} )$. In a fixed probability space $(\Omega,\mathcal{F},\mathbb{P})$ we construct a sequence of $%
n\times n$ Hermitian L\'{e}vy processes $\{X^{(n)}\}_{n\geq1}=\left \{
X_{t}^{(n)}:t>0\right \}_{n\geq 1} $, such that for each $n\geq 1$, the generating triplet $\left( \frac{%
\sigma_{n}^{2}}{n}\mathrm{I}_{n}\otimes \mathrm{I}_{n},\gamma \mathrm{I}%
_{n},\nu_{n}\right) $ is given by

\begin{enumerate}
\item[a)]
\begin{equation}
\sigma_{n}^{2}=\sigma^{2}+\frac{n-1}{n^{2}}.  \label{sigman}
\end{equation}

\item[b)]
\begin{equation}
\gamma= \eta+\int_{|r|\leq1}r\rho(dr)-\int_{|r|>1}\frac{1}{r}\rho(dr) .
\label{gamma}
\end{equation}

\item[c)] The L\'{e}vy measure $\nu_{n}$ is as follows
\begin{equation}
\nu_{n}\left( E\right) =\int_{\mathbb{S}(\mathbb{H}_{n}^{1})}\int
_{\mathbb{R}}1_{E}\left( r\xi \right) n\rho_{n}^{\alpha}\left( \mathrm{d}%
r\right) \pi_{n}\left( \mathrm{d}\xi \right) \text{,\quad}E\in \mathcal{B}%
\left( \mathbb{H}_{n}\backslash \left \{ 0\right \} \right) \text{,}
\label{PDBGCD}
\end{equation}%}	
where:

\begin{enumerate}
\item[i)]  $\pi_{n}$ is a distribution on $\mathbb{S}(%
\mathbb{H}_{n}^{1})$ satisfying%
\begin{equation}
\int \limits_{\mathbb{S}(\mathbb{H}_{n}^{1})}1_{B}\left( \xi \right) \pi
_{n}\left( \mathrm{d}\xi \right) =\int \limits_{\mathbb{S}(\mathbb{C}%
_{n})}1_{\phi_{n}^{-1}\left( B\right) }\left( u\right) \pi \left( \mathrm{d}%
u\right) \text{,\quad}B\in \mathcal{B}\left( \mathbb{S}(\mathbb{H}%
_{n}^{1})\right),  \label{spheres}
\end{equation}
where $\phi_{n}$ denotes the transformation $u\rightarrow uu^{\ast}$ and $\pi$
is the Haar distribution of a random vector in $\mathbb{S}(\mathbb{C}_{n})$,
the unit sphere of $\mathbb{C}_{n}$.

\item[ii)] $\rho_{n}^{\alpha}$ is a measure defined on $\mathbb{R}$ for
	each $n\geq1$ and $\alpha \in(0,1/4)$ by
	\begin{equation}\label{lm}
	\rho_{n}^{\alpha}(dr)=\frac{1+r^{2}}{r^{2}}\rho(dr)
	1_{(1/n,n^{2\alpha}/(n^{\alpha}-1))}(|r|),
	\end{equation}
	where the above expression is understood in the limiting sense when $n=1$.
	Note that $\int \nolimits_{-\infty}^{\infty}(1\wedge r^{2})\rho_{n}^{\alpha
	}(dr)<\infty$.
\end{enumerate}
\end{enumerate}

\begin{remark}
\label{mainR2} i) $X_{t}^{(n)}$ has an absolutely continuous distribution,
for each $t>0$, $n>1$, with respect to the Lebesgue measure on $\mathbb{R}^{2n}$, since $\sigma_{n}^{2}$
in (\ref{sigman}) is non-zero.

ii) The spectrum of $X_{t}^{(n)}$ is simple for each $t>0$, $n>1$, by
the absolute continuity, see \cite{PV}.

iii) $X_{t}^{(n)}$ has a unitary invariant distribution, for each $t>0$, $%
n\geq1$, since the assumptions of Proposition \ref{polar} are satisfied.
This follows since the spherical measure $\pi_{n}$ is a multiple of the Haar
distribution and $\rho_{n}^{\alpha}$ does not depend on $\xi \in \mathbb{S}(%
\mathbb{H}_{n}^{1})$.

iv) The sequence of Hermitian matrices $\left \{ X^{(n)}\right \} _{n\geq1}$
is a L\'{e}vy unitary ensemble.

v) The non-zero jumps of $X^{(n)}$ are of rank one, i.e.,
if $\Delta X^{(n)}_{t}=X^{(n)}_{t}-X^{(n)}_{t-}\neq0$ then
$\Delta X^{(n)}_{t} \in \mathbb{H}_{n}^{1}$.
This follows since the spherical measure $\pi_{n}$ is concentrated on $\mathbb{S}(\mathbb{H}%
_{n}^{1}).$

vi) For fixed $n>1$ and $t>0$, our proposed models differ from the random matrix models given in \cite{BG}, \cite{ca} in the sense that the distribution of $X^{(n)}_t$ always satisfy i) and ii). This is essential in our proof of the functional convergence.
\end{remark}

\textit{Example}. When $Z$ is the free Poisson process, its generating pair is given by: $\eta=-\lambda
$, and $\rho=\lambda \delta_{1}$, with $\lambda>0.$ Then, for each $n>1$%
\begin{equation*}
X_{t}^{(n)}=\sum_{j=1}^{N_{t}}u_{j}u_{j}^{\ast}+\frac{n-1}{n^{2}}B_{t}%
\mathrm{I}_{n}
\end{equation*}
where $\left \{ u_{j}\right \} _{j\geq1}$ is a sequence of independent
uniformly distributed random vectors in $\mathbb{C}_{n}$, $N=\left \{
N_{t}\right \} _{t\geq0}$ is a Poisson process of parameter $\lambda$
independent of $\left \{ u_{j}\right \} _{j\geq1}$ and $B=\left \{
B_{t}\right \} _{t\geq0}$ is a one-dimensional standard Brownian motion
independent of $N$ and $\left \{ u_{j}\right \} _{j\geq1}.$ Then Theorem 1
gives a functional version of the Marchenko--Pastur theorem, in which the
asymptotic noncommutative process is the free Poisson process. We point out
that the functional version in \cite{CG} gives as an asymptotic process the
dilation of the free Poisson distribution, which is not a free L\'{e}vy
process.

\section{The dynamics of the eigenvalues and the measure valued processes}\label{section_semi}

In this section we establish some results needed
later on for the sequence of semimartingales corresponding to the eigenvalues and the
spectral measure valued processes of the ensemble of the Hermitian L\'{e}vy
processes $\left \{ X_{t}^{(n)}:t\geq0\right \} _{n\geq1}$ defined in Section 3.

 Given $A, B \in \mathbb{H}_{n}$ we denote by $\mathcal{D}f(A)(B)$ the G\^ateaux derivative of $f$ at $A$ in the direction $B$, that is
\begin{align*}
 \mathcal{D}f(A)(B)=\frac{d}{dt}\Big|_{t=0}f(A+tB).
\end{align*}
Similarly, we denote by $\mathcal{D}f(A)(B^2)$ to the second G\^ateaux derivative of $f$ at $A$ in the direction $B$ i.e.
\begin{align*}
 \mathcal{D}f(A)(B^2)=\frac{\partial^2}{\partial t^2}\Big|_{t=0}f(A+tB).
\end{align*}
We refer to Chapter V.3 in \cite{BH} and Chapter 6.6 in \cite{HJ} for a proper definition of the G\^ateaux %Fr\'echet
derivative of a function on the set of Hermitian matrices. We first prepare the following lemma.

	\begin{lemma}\label{NI}
		For $s\geq0$ and a continuously differentiable function $f:\mathbb{R}\to\mathbb{R}$,
		\begin{align*}
			\int_{\mathbb{S}(\mathbb{C}_{n})}\sum_{m=1}^{n}\Big [f^{\prime}(\lambda_{m}^{(n)}(X_{s-}^{(n)}))&\mathrm{tr}
			(D\lambda_{m}^{(n)}(X_{s-}^{(n)})vv^{\ast})\Big ]\pi(dv)=\int_{\mathbb{S}(\mathbb{C}_{n})}\mathrm{tr}\left(\mathcal{D}f(X_{s-})(vv^*)\right)\pi(dv).
		\end{align*}
	\end{lemma}
	\begin{proof}
		We first observe that using (\ref{ED}) and Proposition 4.2.3 in \cite{hp},
		\begin{align}\label{aux_n1_2}
			\int_{\mathbb{S}(\mathbb{C}_{n})}\sum_{m=1}^{n}\Big [f^{\prime}(\lambda_{m}^{(n)}(X_{s-}^{(n)}))&\text{tr}
			(D\lambda_{m}^{(n)}(X_{s-}^{(n)})vv^{\ast})\Big ]\pi(dv)\notag\\
			&=\int_{\mathbb{S}(\mathbb{C}_{n})}\sum_{m=1}^{n}\Big [f^{\prime}(\lambda_{m}^{(n)}(X_{s-}^{(n)}))\sum_{i=1}^n\sum_{k=1}^n(D\lambda_{m}^{(n)}(X_{s-}^{(n)}))_{ik}v_k(s-)\overline{v}_i(s-)\Big ]\pi(dv)\notag\\
			&=\frac{1}{n}\sum_{m=1}^{n}\Big [f^{\prime}(\lambda_{m}^{(n)}(X_{s-}^{(n)}))\sum_{i=1}^n(D\lambda_{m}^{(n)}(X_{s-}^{(n)}))_{ii}\Big ]\notag\\
			&=\frac{1}{n}\sum_{m=1}^{n}\Big [f^{\prime}(\lambda_{m}^{(n)}(X_{s-}^{(n)}))\sum_{i=1}^n|u_{im}(s-)|^2\Big ]\notag\\
			&=\frac{1}{n}\sum_{m=1}^{n}f^{\prime}(\lambda_{m}^{(n)}(X_{s-}^{(n)})).
		\end{align}		
		Next, using identity (V.13) in \cite{BH} (see also Theorem 6.6.30(1) in \cite{HJ}), we obtain
		\begin{align*}
			\mathrm{tr}\left(\mathcal{D}f(X_{s-})(vv^*)\right)=\mathrm{tr}\left(U_{s-}[f^{[1]}(\Lambda_{s-})\circ(U^*_{s-}vv^*U_{s-})]U^*_{s-}\right),
		\end{align*}
		where $\circ$ denotes the Schur-product, and $f^{[1]}(\Lambda_{s-})$ is the matrix with entries given by
		\begin{equation*}
			(f^{[1]}(\Lambda_{s-}))_{ij}=
			\begin{cases} \displaystyle \frac{f(\lambda_{i}^{(n)}(X_{s-}^{(n)}))-f(\lambda_{j}^{(n)}(X_{s-}^{(n)}))}{\lambda_{i}^{(n)}(X_{s-}^{(n)})-\lambda_{j}^{(n)}(X_{s-}^{(n)})} &\mbox{if } i\not=j\\
				f^{\prime}(\lambda_{i}^{(n)}(X_{s-}^{(n)})) & \mbox{if } i=j. \end{cases}
		\end{equation*}
		Since $U^*_{s-}v$ and $v$ have the same distribution under $\pi$
		\begin{align}\label{aux_n1_3}
			 \int_{\mathbb{S}(\mathbb{C}_{n})}\mathrm{tr}\left(\mathcal{D}f(X_{s-})(vv^*)\right)\pi(dv)&=\int_{\mathbb{S}(\mathbb{C}_{n})}\mathrm{tr}\left(f^{[1]}(\Lambda)\circ(vv^*)\right)\pi(dv)\notag\\
			&=\int_{\mathbb{S}(\mathbb{C}_{n})}\sum_{m=1}^nf^{\prime}(\lambda_{m}^{(n)}(X_{s-}^{(n)})|v_m|^2\pi(dv)\notag\\
			&=\frac{1}{n}\sum_{m=1}^{n}f^{\prime}(\lambda_{m}^{(n)}(X_{s-}^{(n)})).
		\end{align}
		The result now follows from \eqref{aux_n1_2} and \eqref{aux_n1_3}.
	\end{proof}
	Throughout this paper we will use the notation
\begin{equation*}
	\langle \mu,f\rangle:=\int_{\mathbb{R}}f(x)\mu(dx),
\end{equation*}
for any bounded measurable function $f$ and $\mu \in \mathrm{Pr}(\mathbb{R})$. We denote the set of functions $f:\mathbb{R}\to\mathbb{R}$ that are $k$ times continuously differentiable with bounded derivatives by $\mathcal{C}^{k}_{b}(\mathbb{R})$, for $k=1,2$.

\begin{proposition}\label{IE} Let $\left \{ \mu_{t}^{(n)}:t\geq0\right \}_{n\geq1}$
be the spectral measure-valued process (%
\ref{mvp}) of $\left \{ X_{t}^{(n)}:t\geq0\right \} _{n\geq1}.$
Then, for each $n\geq1$, $t\geq0$, and $f\in \mathcal{C}^{2}_{b}(\mathbb{R})$
\begin{align}
	& \langle \mu_{t}^{(n)},f\rangle=\langle \mu_{0}^{n},f\rangle+M_{t}^{n,f}+{\frac{\sigma
			_{n}^{2}}{2}}\int_{0}^{t}\int_{\mathbb{R}^{2}}\frac{f^{\prime}(x)-f^{\prime
		}(y)}{x-y}\mu_{s}^{n}(dx)\mu_{s}^{n}(dy)ds+{\gamma}\int_{0}^{t}\int _{%
		\mathbb{R}}f^{\prime}(x)\mu_{s}^{n}(dx)ds  \notag \\
	& +\int_{0}^{t}\int_{\mathbb{S}(\mathbb{C}_{n})}\int _{\mathbb{%
					R}}\mathrm{tr}\Big [f(X_{s-}^{(n)}+rvv^{\ast})-f(X_{s-}^{(n)})-\mathcal{D}f(X_{s-})(rvv^*)1_{\{|r|\leq1\}}\Big ]\rho
			_{n}^{\alpha}(dr)\pi(dv)ds.  \label{SDEaa}
\end{align}
\end{proposition}

\begin{proof}
For each $n\geq1$, and $t\geq0$ let $\lambda_{1}^{(n)}(t)>\cdots>\lambda_{n}^{(n)}(t)$
denote the eigenvalues of $X^{(n)}_{t}$. Then, by Proposition \ref{PVeigen},
\begin{align}
\lambda_{m}^{(n)}&(X_{t}^{(n)})  =\lambda_{m}(X_{0}^{(n)})+\gamma \sum
_{i=1}^{n}\int_{0}^{t}(D\lambda_{m}^{(n)}(X_{s-}^{(n)}))_{ii}ds+\frac {%
\sigma_{n}^{2}}{n}\int_{0}^{t}\sum_{j\not =m}\frac{1}{(\lambda_{m}^{(n)}-%
\lambda_{j}^{(n)})(s-)}ds+M_{t}^{n,m}  \notag
\label{SDE} \\
& +\int_{0}^{t}\int_{\mathbb{S}(\mathbb{H}_{n}^{1})}\int_{\mathbb{R}}\left[
\lambda_{m}^{(n)}(X_{s-}^{(n)}+ry)-\lambda_{m}^{(n)}(X_{s-}^{(n)})-\mathrm{tr}%
(D\lambda_{m}^{(n)}(X_{s-}^{(n)})ry\right] 1_{\{|r|\leq1\}}) n\rho_{n}^{%
\alpha}(dr)\pi_{n}(dy)ds,
\end{align}
where the process $(M_{t}^{n,m})_{t\geq0}$ is a martingale.

From (\ref{SDE})
and an application of It\^{o}'s formula we obtain that
\begin{align}
&\langle  \mu_{t}^{(n)},f\rangle=\langle \mu_{0}^{n},f\rangle+	M_{t}^{n,f}+\frac{\sigma_{n}^{2}}{n^{2}}%
\sum_{m=1}^{n}\int_{0}^{t}f^{\prime}(\lambda_{m}^{(n)}(X_{s-}^{(n)}))\sum_{j%
\not =m}\frac{1}{(\lambda_{m}^{(n)}-\lambda_{j}^{(n)})(s-)}ds \notag
\label{ito3} \\
& +\frac{\gamma}{n}\sum_{m=1}^{n}\sum_{i=1}^{n}\int_{0}^{t}f^{%
\prime}(\lambda_{m}^{(n)}(X_{s-}^{(n)}))(D%
\lambda_{m}^{(n)}(X_{s-}^{(n)}))_{ii}ds+\frac{1}{2n}\sum_{m=1}^{n}%
\int_{0}^{t}f^{\prime \prime }(\lambda_{m}^{(n)}(X_{s-}^{(n)}))d\langle
M^{n,m},M^{n,m}\rangle_{s}^{c}  \notag \\
& +\sum_{m=1}^{n}\int_{0}^{t}f^{\prime}(%
\lambda_{m}^{(n)}(X_{s-}^{(n)}))\int_{\mathbb{S}(\mathbb{H}_{n}^{1})}\int_{%
\mathbb{R}}\Big[ \lambda_{m}^{(n)}(X_{s-}^{(n)}+ry)-%
\lambda_{m}^{(n)}(X_{s-}^{(n)})-\mathrm{tr}(D%
\lambda_{m}^{(n)}(X_{s-}^{(n)})ry)1_{\{|r|\leq1\}}\Big]\rho_{n}^{%
\alpha}(dr)\pi_{n}(dy)ds  \notag \\
& +\sum_{m=1}^{n}\int_{0}^{t}\int_{\mathbb{S}(\mathbb{H}%
_{n}^{1})}\int_{\mathbb{R}}\left[ f(\lambda_{m}^{(n)}(X_{s-}^{(n)}+ry))-f(%
\lambda
_{m}^{(n)}(X_{s-}^{(n)}))-f^{\prime}(\lambda_{m}^{(n)}(X_{s-}^{(n)}))\Delta
\lambda_{m}^{(n)}(X_{s}^{(n)})\right]\rho_{n}^{\alpha}(dr)\pi _{n}(dy)ds,
\end{align}
where $(M_{t}^{n,f})_{t\geq0}$ is a local martingale, given by
\begin{align}\label{mart}
	M_{t}^{n,f} & ={\frac{\sigma_{n}}{n^{3/2}}\sum_{m=1}^{n}\int_{0}^{t}f^{\prime}(\lambda_{m}^{(n)}(X_{s-}^{(n)}))dW_{s}^{m}}  +	 \frac{1}{n}\int_{0}^{t}\int_{\mathbb{S}(\mathbb{H}_{n}^{1})}\int_{\mathbb{R}}\mathrm{tr}\Big [f(X_{s-}^{(n)}+ry)-f(X_{s-}^{(n)})\Big ]
	\widetilde {J}_{X}(ds,dr,dy),
\end{align}
and $W^{m},m=1,...,n$ are independent one-dimensional Brownian motions.%}
\par Following Remark \ref{mart rep} it is easy to check that
\begin{equation*}
\frac{1}{2n}\sum_{m=1}^{n}\int_{0}^{t}f^{\prime
\prime}(\lambda_{m}^{(n)}(X_{s-}^{(n)})){d\langle
M^{n,m},M^{n,m}\rangle_{s}^{c}}={\frac {\sigma_{n}^{2}}{2n^{2}}}%
\sum_{m=1}^{n}\int_{0}^{t}f^{\prime
\prime}(\lambda_{m}^{(n)}(X_{s-}^{(n)}))ds,
\end{equation*}
and therefore{\
\begin{align}
\frac{\sigma_{n}^{2}}{n^{2}}\sum_{m=1}^{n} &
\int_{0}^{t}f^{\prime}(\lambda_{m}^{(n)}(X_{s-}^{(n)}))\sum_{j\not =m}\frac{1%
}{(\lambda_{m}^{(n)}-\lambda_{j}^{(n)})(s-)}ds+\frac{1}{2n}%
\sum_{m=1}^{n}\int_{0}^{t}f ^{\prime
\prime}(\lambda_{m}^{(n)}(X_{s-}^{(n)}))d\langle
M^{n,m},M^{n,m}\rangle_{s}^{c}  \notag \\
& =\frac{\sigma_{n}^{2}}{2}\int_{0}^{t}\int_{\mathbb{R}^{2}}\frac{f^{\prime
}(x)-f^{\prime}(y)}{x-y}\mu_{s}^{n}(dx)\mu_{s}^{n}(dy)ds.  \notag
\end{align}
}  The drift term in (\ref{ito3}) is expressed, using (\ref{ED}), as
{\
\begin{align*}
\frac{\gamma}{n}\sum_{m=1}^{n}\sum_{i=1}^{n}\int_{0}^{t}f^{\prime}(\lambda
_{m}^{(n)}(X_{s-}^{(n)}))(D\lambda_{m}^{(n)}(s-))_{ii}ds & =\frac{\gamma}{n}%
\sum_{m=1}^{n}\sum_{i=1}^{n}\int_{0}^{t}f^{\prime}(%
\lambda_{m}^{(n)}(X_{s-}^{(n)}))|\bar{u}_{im}^{(n)}u_{im}^{(n)}|ds \\
& =\frac{\gamma}{n}\sum_{m=1}^{n}\int_{0}^{t}f^{\prime}(%
\lambda_{m}^{(n)}(X_{s-}^{(n)}))ds=\gamma \int_{0}^{t}\int_{\mathbb{R}%
}f^{\prime}(x)\mu_{s}^{(n)}(dx)ds.
\end{align*}
} The last two terms in (\ref{ito3}) can be written, using (\ref{PDBGCD}) and (\ref{spheres}), as
\begin{align*}
\sum_{m=1}^{n}\int_{0}^{t}\int_{\mathbb{S}(\mathbb{C}_{n})}\int_{\mathbb{R}}%
\Bigg[f(\lambda_{m}^{(n)} &
(X_{s-}^{(n)}+rvv^{\ast}))-f(\lambda_{m}^{(n)}(X_{s-}^{(n)})) \\
& -f^{\prime}(\lambda_{m}^{(n)}(X_{s-}^{(n)}))\mathrm{tr}(D%
\lambda_{m}^{(n)}(X_{s-}^{(n)})rvv^{\ast})1_{\{|r|\leq1\}}\Bigg]%
\rho_{n}^{\alpha}(dr)\pi(dv)ds.
\end{align*}
Thus (\ref{ito3}) is expressed as
\begin{align}
& \langle \mu_{t}^{(n)},f\rangle=\langle \mu_{0}^{n},f\rangle+M_{t}^{n,f}+{\frac{\sigma
_{n}^{2}}{2}}\int_{0}^{t}\int_{\mathbb{R}^{2}}\frac{f^{\prime}(x)-f^{\prime
}(y)}{x-y}\mu_{s}^{n}(dx)\mu_{s}^{n}(dy)ds+{\gamma}\int_{0}^{t}\int _{%
\mathbb{R}}f^{\prime}(x)\mu_{s}^{n}(dx)ds  \notag \\
& +\int_{0}^{t}\int_{\mathbb{S}(\mathbb{C}_{n})}\int _{\mathbb{%
R}}\sum_{m=1}^{n}\Big [f(\lambda_{m}^{(n)}(X_{s-}^{(n)}+rvv^{\ast}))-f(\lambda
_{m}^{(n)}(X_{s-}^{(n)}))  \notag \\
& \hspace{6cm}-f^{\prime}(\lambda_{m}^{(n)}(X_{s-}^{(n)}))\mathrm{tr}%
(D\lambda_{m}^{(n)}(X_{s-}^{(n)})rvv^{\ast})1_{\{|r|\leq1\}}\Big ]\rho
_{n}^{\alpha}(dr)\pi(dv)ds.  \label{SDEa}
\end{align}
Using the fact that the last term in \eqref{SDEa} is integrable, we obtain (\ref{SDEaa}) by an application of Fubini's theorem together with Lemma \ref{NI}.
\end{proof}

In the rest of this section we will study the convergence of the martingale term, to this end we need the following auxiliary result.
\begin{lemma}\label{der_bound}
Let $A\in\mathbb{H}_n$. 

(i) For any $f\in \mathcal{C}^{1}_{b}(\mathbb{R})$,
    \begin{align*}
		\int_{\mathbb{S}(\mathbb{C}_{n})}\left[\mathrm{tr}\left(\mathcal{D}f(A)(vv^{\ast})\right)\right]^2\pi(dv)\leq \|f'\|_{\infty}^2.
	\end{align*}

(ii) For any $f\in \mathcal{C}^{2}_{b}(\mathbb{R})$,
	\begin{align*}
		\int_{\mathbb{S}(\mathbb{C}_{n})}|\mathrm{tr}\left(\mathcal{D}^2f(A)\left((vv^{\ast})^2\right)\right)|\pi(dv)\leq 2\|f''\|_{\infty},
	\end{align*}
and
\begin{align*}
	|\mathrm{tr}\left(\mathcal{D}^2f(A)\left((n^{-1}\mathrm{I_n})^2\right)\right)|\leq 2\|f''\|_{\infty}.
\end{align*}
	\end{lemma}
\begin{proof}
	(i) Let $A=U\Lambda U^*$, where $\Lambda$ is a diagonal matrix, then by identity (V.13) in \cite{BH} (see also Theorem 6.6.30(1) in \cite{HJ})
together with the fact that
	$U^*U=I$
	\begin{align}\label{iden_fd}
		 \mathrm{tr}\left(\mathcal{D}f(A)(vv^{\ast})\right)=\mathrm{tr}\left(U\left[f^{[1]}(\Lambda)\circ(U^*vv^{\ast}U)\right]U^*\right)=\mathrm{tr}\left(\left[f^{[1]}(\Lambda) \circ(U^*vv^{\ast}U)\right]\right),
	\end{align}
	where $\circ$ denotes the Schur product. Therefore using that $U^*v$
	and $v$ have the same distribution under $\pi$
	\begin{align*}
		 \int_{\mathbb{S}(\mathbb{C}_{n})}\left[\mathrm{tr}\left(\mathcal{D}f(A)(vv^{\ast})\right)\right]^2\pi(dv)=\int_{\mathbb{S}(\mathbb{C}_{n})}\left[\mathrm{tr}\left(f^{[1]}(\Lambda)\circ(vv^{\ast})\right)\right]^2\pi(dv)\leq \|f'\|_{\infty}^2.
	\end{align*}
	 (ii) To obtain the second identity we use Exercise V.3.9 in \cite{BH} (see also Theorem 6.6.30(2) in \cite{HJ})
to note that for a matrix $B\in\mathbb{H}_n$
	\begin{align}\label{seg_der_exp}
	 \mathcal{D}^2f(\Lambda)(B^2)=2\sum_{i,j,k}f^{[2]}(\lambda_i,\lambda_j,\lambda_k)P_iBP_jBP_k,
	\end{align}
  where $P_{i}$ are the projections onto the coordinate axes, $f^{[2]}(\lambda_i, \lambda_j, \lambda_k)$=$(f^{[1]}(\lambda_i, \lambda_j)-f^{[1]}(\lambda_i, \lambda_k))/(\lambda_j-\lambda_k)$ if $\lambda_i, \lambda_j, \lambda_k$ are distinct, and defined by continuity otherwise. By a straightforward computation we have $(P_iBP_jBP_k)_{lm}=(B)_{ij}(B)_{jk}1_{\{l=i,m=k\}}$ for $l,m=1,\dots,n$, this implies that
	\begin{align}\label{trace}
		\mathrm{tr}(P_iBP_jBP_k)=(B)_{ij}(B)_{ji}1_{\{k=i\}}.
	\end{align}
	Therefore, using that $U^*v$
	and $v$ have the same distribution under $\pi$
	\begin{align*}
		 \int_{\mathbb{S}(\mathbb{C}_{n})}|\mathrm{tr}\left(\mathcal{D}^2f(A)\left((vv^{\ast})^2\right)\right)|\pi(dv)&=\int_{\mathbb{S}(\mathbb{C}_{n})}|\mathrm{tr}\left(U\left[\mathcal{D}^2f(\Lambda)\left((U^*vv^{\ast}U)^2\right)\right]U^*\right)|\pi(dv)\\
		 &=2\int_{\mathbb{S}(\mathbb{C}_{n})}\left|\sum_{i,j}f^{[2]}(\lambda_i,\lambda_j,\lambda_i)(vv^{\ast})_{ij}(vv^{\ast})_{ji}\right|\pi(dv)\\&\leq2\|f''\|_{\infty}\int_{\mathbb{S}(\mathbb{C}_{n})}\sum_{i,j}|v_i|^2|v_j|^2\pi(dv)=2\|f''\|_{\infty}.
	\end{align*}
For the remaining identity we note that using \eqref{seg_der_exp} together with \eqref{trace} we obtain
\begin{align*}
|\mathrm{tr}\left(\mathcal{D}^2f(A)\left((n^{-1}\mathrm{I_n})^2\right)\right)|\leq\frac{2}{n}\sum_{i}|f^{[2]}(\lambda_i,\lambda_j,\lambda_i)|\leq 2\|f''\|_{\infty}.
\end{align*}
\end{proof}

\begin{lemma}
\label{martl} For any $f\in\mathcal{C}^2_b(\mathbb{R})$,
the martingale $(M_{t}^{n,f})_{t\geq0}$ in (\ref{mart}) satisfies
\begin{equation*}
\lim_{n\rightarrow \infty}\sup_{0\leq t\leq T}|M_{t}^{n,f}|=0\qquad \text{in
probability,}
\end{equation*}
for any $T>0$.
\end{lemma}

\begin{proof}
We show the convergence of each term in (\ref{mart}). Let $\varepsilon>0$.
By Doob's inequality applied to the first term on the right hand side of (\ref{mart})
{\
\begin{equation*}
\mathbb{P}\left( \sup_{0\leq t\leq T}\left \vert \frac{\sigma_{n}}{n^{3/2}}%
\int_{0}^{t}\sum_{m=1}^{n}f^{\prime}(%
\lambda_{m}^{(n)}(X_{s-}^{(n)}))dW_{s}^{m}\right \vert >\varepsilon \right)
\leq \frac{1}{\varepsilon^{2}}\frac{{\sigma_{n}}^{2}}{n^{3}}\mathbb{E}\left[
\left(
\int_{0}^{T}\sum_{m=1}^{n}f^{\prime}(%
\lambda_{m}^{(n)}(X_{s-}^{(n)}))dW_{s}^{m}\right) ^{2}\right]
\end{equation*}
}%
\begin{equation*}
\leq \frac{1}{\varepsilon^{2}}\frac{{\sigma_{n}}^{2}}{n^{3}}\mathbb{E}\left[
\int_{0}^{T}\sum_{m=1}^{n}(f^{\prime}(\lambda_{m}^{(n)}(X_{s-}^{(n)})))^{2}ds%
\right] \leq \frac{1}{\varepsilon^{2}}\frac{\sigma^2+1}{n^{2}}\Vert
f^{\prime}\Vert_{\infty}^{2}T,
\end{equation*}
which converges to $0$ as $n\rightarrow \infty$.

Let us consider $n>1$, then by an application of Doob's inequality in the
second term on the right hand side of (\ref{mart}) together with the mean value theorem and Lemma \ref{der_bound}(i),

\begin{align*}
	\mathbb{P} &  \left(  \sup_{0\leq t\leq T}\left \vert \frac{1}{n}\int_{0}^{t}\int_{\mathbb{S}(\mathbb{H}_{n}^{1})}\int_{\mathbb{R}}\mathrm{tr}\Big [f(X_{s-}^{(n)}+ry)-f(X_{s-}^{(n)})\Big ] \widetilde{J}_{X}(ds,dr,dy)\right \vert >\varepsilon \right)
	\\
	&  \leq \frac{1}{n^{2}}\frac{1}{\varepsilon^{2}}\mathbb{E}\left[  \left(
	\int_{0}^{t}\int_{\mathbb{S}(\mathbb{H}_{n}^{1})}\int
	_{\mathbb{R}}\mathrm{tr}\Big [f(X_{s-}^{(n)}+ry)-f(X_{s-}^{(n)})\Big ]  \widetilde{J}_{X}(ds,dr,dy)\right)  ^{2}\right]
	\\
  &=\frac{1}{n}\frac{1}{\varepsilon^{2}}\mathbb{E}\int_{0}^{T}\int_{\mathbb{R}}\int
	_{\mathbb{S}(\mathbb{C}_{n})}\mathrm{tr}\left(\left[
	f(X_{s-}^{(n)}+rvv^{\ast})-f(X_{s-}^{(n)})\right]\right)^{2}\pi(dv)\rho
	_{n}^{\alpha}(dr)ds\\
	&  \leq \Vert f^{\prime}\Vert_{\infty}^{2}\frac{1}{n}\frac{1}{\varepsilon^{2}%
	}T\int_{\mathbb{R}}r^{2}\rho_{n}^{\alpha}(dr)\leq \left(\frac{1}{n}+\frac{n^{4\alpha}}{n(n^{\alpha}-1)^{2}}\right)  C(f,T),
\end{align*}
for some constant $C(f,T)>0$. Hence $ \left(\frac{1}{n}+\frac{n^{4\alpha}}{n(n^{\alpha}-1)^{2}}\right)C(f,T)\rightarrow0$ as $n\rightarrow \infty$. This concludes the proof.
\end{proof}

\section{Tightness}

Let $\{(\mu_{t}^{(n)})_{t\geq0}:n\geq1\}$ be the family of measure
valued-processes (\ref{mvp}) of the Hermitian L\'{e}vy process ensemble
$\left(  X^{(n)}\right)  _{n\geq1}$ introduced in Section \ref{sectionHL}. In this section
we prove that this family is tight in the space $\mathcal{D}(\mathbb{R}%
_{+},\mathrm{Pr}(\mathbb{R}))$. We denote the set of functions $f:\mathbb{R}\to\mathbb{R}$ that are twice
differentiable with compact support by $\mathcal{C}^{2}_{c}(\mathbb{R})$.

The keys to the proof are the eigenvalue semimartingale estimates of Section
\ref{section_semi}, the fact that for each $n\geq1$ all the jumps of the Hermitian L\'{e}vy
process $\left\{  X_{t}^{(n)}:t\geq0\right\}  $ are of rank one and Lemma \ref{der_bound}.

\begin{proposition}
	\label{tightness} The family of measures $\{(\mu_{t}^{(n)})_{t\geq0}:n\geq1\}$
	is tight in the space $\mathcal{D}(\mathbb{R}_{+},\mathrm{Pr}(\mathbb{R}))$.
\end{proposition}

\begin{proof}
	First we will start by establishing that the family of measure
		valued-processes $\{(\mu_{t}^{(n)})_{t\geq
		0}:n\geq1\}$ is tight in the space $\mathcal{D}(\mathbb{R}_{+},\mathrm{Pr}(\mathbb{R}))$, when
	$\mathrm{Pr}(\mathbb{R})$ is endowed with the topology of the vague convergence.
		To this end, we use
	the Aldous--Rebolledo Criterion (see \cite{al}, \cite{E}, \cite{rc}) to prove
    that for each $f\in\mathcal{C}_{c}^{2}(\mathbb{R})$ the sequence of real
	processes $\{(\langle\mu_{t}^{(n)},f\rangle)_{t\geq0}:n\geq1\}$ is tight. We
	split the proof of the tightness of the semimartingale $\langle\mu
	^{(n)},f\rangle$ into two steps: the first is on the bounded variation part
	and the second on the martingale part.
	
	For any $f\in\mathcal{C}_{c}^{2}(\mathbb{R})$ we have from (\ref{SDEaa}) that
   	\begin{align}
		&  \langle\mu_{t}^{(n)},f\rangle-\langle\mu_{s}^{(n)},f\rangle={\frac
			{\sigma_{n}^{2}}{2}}\int_{s}^{t}\int_{\mathbb{R}^{2}}\frac{f^{\prime
			}(x)-f^{\prime}(y)}{x-y}\mu_{u}^{(n)}(dx)\mu_{u}^{(n)}(dy)du+{\gamma}\int%
		_{s}^{t}\int_{\mathbb{R}}f^{\prime}(x)\mu_{u}^{(n)}(dx)du\nonumber\\
		&+M_{t}^{n,f}-M_{s}^{n,f}  + \int_{s}^{t}\int_{\mathbb{S}(\mathbb{C}_{n})}\int%
		_{\mathbb{R}}\text{tr}\Big [f(X_{u-}^{(n)}+rvv^{\ast})-f(X_{u-}^{(n)})-\mathcal{D}f(X_{u-})(rvv^*)1_{\{|r|\leq1\}}\Big ]\rho
		_{n}^{\alpha}(dr)\pi(dv)du.\nonumber
	\end{align}
	Let us denote by $V^{n,f}$ the bounded variation part of the semimartingale
	$\langle\mu^{(n)},f\rangle$. Then, for each $0\leq s\leq t$,%
	\begin{align}
		&  V_{t}^{n,f}-V_{s}^{n,f}={\frac{\sigma_{n}^{2}}{2}}\int_{s}^{t}%
		\int_{\mathbb{R}^{2}}\frac{f^{\prime}(x)-f^{\prime}(y)}{x-y}\mu_{u}%
		^{(n)}(dx)\mu_{u}^{(n)}(dy)du+{\gamma}\int_{s}^{t}\int_{\mathbb{R}}f^{\prime}(x)\mu_{u}%
		^{(n)}(dx)du\nonumber\\
		&  +\int_{s}^{t}\int_{\mathbb{S}(\mathbb{C}_{n})}\int%
		_{\mathbb{R}}\text{tr}\Big [f(X_{u-}^{(n)}+rvv^{\ast}))-f(X_{u-}^{(n)})-\mathcal{D}f(X_{u-})(rvv^*)1_{\{|r|\leq1\}}\Big ]\rho
		_{n}^{\alpha}(dr)\pi(dv)du.\label{SDE3}%
	\end{align}
	
	Let $\delta>0$ and $\theta\in\lbrack0,\delta]$. Let $T^{\prime}>0$ and let
	$(\tau_{n})_{n\geq1}$ be a sequence of stopping times such that $0\leq\tau
	_{n}<T^{\prime}$.
	
	Next we estimate each term in the right hand of (\ref{SDE3}).
	By the mean value theorem we have for any $x,y\in\mathbb{R}$
	\begin{align*}
	|f'(x)-f'(y)|\leq \|f''\|_{\infty}|x-y|.
	\end{align*}
    Hence, for the first term in (\ref{SDE3}) we obtain
		\begin{equation}
	\left\vert {\sigma_{n}}\int_{\tau_{n}}^{\tau_{n}+\theta}%
	\int_{\mathbb{R}^{2}}\frac{f^{\prime}(x)-f^{\prime}(y)}{x-y}\mu_{u}%
	^{(n)}(dx)\mu_{u}^{(n)}(dy)du\right\vert \leq(\sigma^{2}+1)^{1/2}\Vert
	f^{\prime\prime}\Vert_{\infty}\delta.\label{var1}%
	\end{equation}
	For the second term in the right hand of (\ref{SDE3}), 
	\begin{equation}
	\left\vert \int_{\tau_{n}}^{\tau_{n}+\theta}\int_{\mathbb{R}}f^{\prime}(x) \mu_{u}
	^{(n)}(dx)du\right\vert\leq\Vert f^{\prime}\Vert_{\infty}\delta.\label{var2}%
	\end{equation}
	
	For the jump part in (\ref{SDE3}), we can apply Fubini's theorem to obtain
	\begin{align}\label{aux_ni_2}
\int_{\tau_{n}}^{\tau_{n}+\theta}&\int_{\mathbb{S}(\mathbb{C}_{n})}\int
		_{\mathbb{R}}\mathrm{tr}\Big [f(X_{u-}^{(n)}+rvv^{\ast})-f(X_{u-}^{(n)})-\mathcal{D}f(X_{u-})(rvv^*)1_{\{|r|\leq1\}}\Big ]\rho
		_{n}^{\alpha}(dr)\pi(dv)du\notag\\
&=\int_{\tau_{n}}^{\tau_{n}+\theta}\int
		_{|r|\geq1}\int_{\mathbb{S}(\mathbb{C}_{n})}\mathrm{tr}\Big [f(X_{u-}^{(n)}+rvv^{\ast})-f(X_{u-}^{(n)})\Big ]\pi(dv)\rho
		_{n}^{\alpha}(dr)du\notag\\
&+\int_{\tau_{n}}^{\tau_{n}+\theta}\int
		_{|r|\leq1}\int_{\mathbb{S}(\mathbb{C}_{n})}\mathrm{tr}\Big [f(X_{u-}^{(n)}+rvv^{\ast})-f(X_{u-}^{(n)})-\mathcal{D}f(X_{u-})(rvv^*)1_{\{|r|\leq1\}}\Big ]\pi(dv)\rho
		_{n}^{\alpha}(dr)du.
\end{align}
For the first term in the right hand side of  (\ref{aux_ni_2}) we use Lemma III.5 in \cite{ca} and the fact that $f$ has compact support to obtain
\begin{align}\label{var4prime}
	\Bigg|\int_{\tau_{n}}^{\tau_{n}+\theta}\int
	_{|r|\geq1}&\int_{\mathbb{S}(\mathbb{C}_{n})}\text{tr}\Big [f(X_{u-}^{(n)}+rvv^{\ast})-f(X_{s-}^{(n)})\Big ]\pi(dv)\rho
	_{n}^{\alpha}(dr)du\Bigg|\notag\\&\leq\|f\|_{1}^{\prime}\int_{\tau_{n}}^{\tau_{n}+\theta}\int
	_{|r|\geq1}\rho
	_{n}^{\alpha}(dr)du\leq C_1(f)\delta,
\end{align}	
where $C_{1}(f)>0$ is a constant and $\|g\|_1'=\sup_{\overset{n\geq 1}{x_1\leq y_1\leq x_2\leq \dots\leq y_n}}\sum_{i=1}^n|g(y_i)-g(x_i)|$ for any $g\in\mathcal{C}_c^2(\mathbb{R})$. 
For the second term in the right hand of \eqref{aux_ni_2} we use
 Lemma \ref{der_bound}(ii) and  Taylor's theorem to obtain
	\begin{align}\label{var5}
		\Bigg|\int_{\tau_{n}}^{\tau_{n}+\theta}&\int
		_{|r|\leq1}\int_{\mathbb{S}(\mathbb{C}_{n})}\mathrm{tr}\Big [f(X_{u-}^{(n)}+rvv^{\ast})-f(X_{u-}^{(n)})-\mathcal{D}f(X_{u-})(rvv^*)1_{\{|r|\leq1\}}\Big]\pi(dv)\rho_{n}^{\alpha}(dr)du\Bigg|\notag\\
		&=\Bigg|\int_{\tau_{n}}^{\tau_{n}+\theta}\int
		 _{|r|\leq1}\int_{\mathbb{S}(\mathbb{C}_{n})}\int_0^1r^2(1-\xi)\left[\mathrm{tr}\left(\mathcal{D}^2f(X_{u-}^{(n)}+r\xi vv^*)\left((vv^{\ast})^2\right)\right)\right]d\xi\pi(dv)\rho_{n}^{\alpha}(dr)du\Bigg|\notag\\
		&\leq \|f''\|_{\infty}\int_{\tau_{n}}^{\tau_{n}+\theta}\int_{|r|\leq 1}r^2\rho
		_{n}^{\alpha}(dr)du\leq C_2(f)\delta,
	\end{align}
	where $C_{2}(f)>0$ is a constant.

From (\ref{var1}), (\ref{var2}), (\ref{var4prime}) and (\ref{var5})
we conclude that there exists a constant $K_{1}>0$, which does not depend on
$n\geq1$, such that
\begin{equation}
\sup_{n\geq1}\sup_{\theta\in\lbrack0,\delta]}\mathbb{E}\left[  \left\vert
V_{\tau_{n}+\theta}^{n,f}-V_{\tau_{n}}^{n,f}\right\vert \right]  <K_{1}\delta  .\label{varti}%
\end{equation}

The next step is to prove the tightness of the laws of the martingale part of
the semimartingale $\langle\mu^{(n)},f\rangle$. Recall that the quadratic
variation of the martingale $M^{n,f}$ is given by
\begin{align}
	\langle M^{n,f},M^{n,f}\rangle_{t} &  =\frac{\sigma_{n}^{2}}{n^{3}}\int%
	_{0}^{t}\sum_{m=1}^{n}(f^{\prime}(\lambda_{m}^{(n)}(X_{s-}^{(n)}%
	)))^{2}ds\nonumber\\
	&  +\frac{1}{n}\int_{0}^{t}\int_{\mathbb{S}(\mathbb{C}_{n})}%
	\int_{\mathbb{R}}\left(\mathrm{tr}\left[
	f(X_{s-}^{(n)}+rvv^{\ast})-f(X_{s-}^{(n)})\right] \right)^{2}
	\rho_{n}^{\alpha}(dr)\pi(dv)ds.\nonumber
\end{align}
For the first term, note that
\begin{equation}
\frac{{\sigma_{n}}^{2}}{n^{3}}\mathbb{E}\left[  \int_{\tau_{n}}^{\tau
	_{n}+\theta}\sum_{m=1}^{n}(f^{\prime}(\lambda_{m}^{(n)}(X_{s-}^{(n)}%
)))^{2}ds\right]  \leq\frac{{(\sigma+1)}^{2}}{n^{2}}\Vert f^{\prime}%
\Vert_{\infty}^{2}\delta.\label{martti1}%
\end{equation}
For the second term, by the proof of Lemma \ref{martl} together with Lemma III.5 in \cite{ca}, one obtains for $n\geq 1$
\begin{align}
	\frac{1}{n}   \mathbb{E}\Bigg[  \int_{\tau_{n}}^{\tau_{n}+\theta}%
	&\int_{\mathbb{S}(\mathbb{C}_{n})}\int_{\mathbb{R}}\left(\mathrm{tr}\left[
	f(X_{s-}^{(n)}+rvv^{\ast})-f(X_{s-}^{(n)})\right] \right)^{2}\rho_{n}^{\alpha}%
	(dr)\pi(dv)ds\Bigg]  \nonumber\label{marti2}\\
	&  \leq\frac{\delta}{n}\left(\Vert f'\Vert_{\infty}^{2}\int_{|r|\leq 1}(1+r^{2})%
	\rho(dr)+(\|f\|_1')^2\int_{|r|> 1}\frac{1+r^2}{r^2}
	\rho(dr)\right)\leq\delta C_{3}(f)\text{,}%
\end{align}
for some constant $C_{3}(f)>0$. From (\ref{martti1}) and (\ref{marti2})
there exists a constant $K_{2}>0$ independent of $n\geq1$ such that
\begin{equation}
\sup_{n\geq1}\sup_{\theta\in\lbrack0,\delta]}\mathbb{E}\left[  \left\vert
\langle M^{n,f},M^{n,f}\rangle_{\tau_{n}+\theta}-\langle M^{n,f}%
,M^{n,f}\rangle_{\tau_{n}}\right\vert \right]  <\delta K_{2}.\label{martii}%
\end{equation}

Let us fix $T>0$. Then proceeding as in the first part of the proof, it can be
seen that there exists a constant $K_{1}(T)>0$ depending on $T$ such that
\[
\sup_{n\geq1}\mathbb{E}\left[  \left(  \sup_{t\in\lbrack0,T]}V_{t}%
^{n,f}\right)  ^{2}\right]  <K_{1}(T).
\]
On the other hand, from the proof of Lemma \ref{martl}, there exists a
constant $K_{2}(T)>0$, that depends on $T$, such that
\[
\sup_{n\geq1}\mathbb{E}\left[  \left(  \sup_{t\in\lbrack0,T]}M_{t}%
^{n,f}\right)  ^{2}\right]  <K_{2}(T).
\]
Therefore there exists a constant $K(T)>0$ depending on $T$ such that
\begin{equation}
\sup_{n\geq1}\mathbb{E}\left[  \left(  \sup_{t\in\lbrack0,T]}\langle\mu_t
^{(n)},f\rangle\right)  ^{2}\right]  <K(T).\label{tight}%
\end{equation}

Now, from (\ref{varti}), (\ref{martii}) and (\ref{tight}), we can use the
Aldous--Rebolledo criterion (see \cite{al}, \cite{E}, \cite{rc}) to conclude
that the sequence of real processes $\{(\langle\mu_{t}^{(n)},f\rangle
)_{t\geq0}:n\geq1\}$ is tight, and that consequently the sequence of processes
$\{(\mu_{t}^{(n)})_{t\geq0}:n\geq1\}$ is tight in the space $\mathcal{D}%
(\mathbb{R}_{+},\mathrm{Pr}(\mathbb{R}))$ with $\mathrm{Pr}(\mathbb{R})$
endowed with the topology of vague convergence.

It remains to extend the above result to the case when $\mathrm{Pr}%
(\mathbb{R})$ is endowed with the topology of weak convergence. Note that
taking $f=1$, the sequence of real-valued processes $\{(\langle\mu_{t}%
^{(n)},f\rangle)_{t\geq0}:n\geq1\}$ is tight. On the other hand note that
Lemma \ref{martl} implies that for any convergent subsequence of $\{(\mu
_{t}^{(n)})_{t\geq0}:n\geq1\}$, the limit is strongly continuous. Therefore by
an application of the M\'{e}l\'{e}ard--Roelly criterion (see \cite{MR}), it
follows that the sequence $\{(\mu_{t}^{(n)})_{t\geq0}:n\geq1\}$ is tight in
the space $\mathcal{D}(\mathbb{R}_{+},\mathrm{Pr}(\mathbb{R})).$
\end{proof}

\section{Characterization of the weak limit of the measure valued-processes}

Let $z\in \mathbb{C}\backslash \mathbb{R}$ and $f_{z}(x)=(z-x)^{-1}$ for $%
x\in \mathbb{R}$. For any continuous function $(\nu_{t})_{t\geq0}\in C(%
\mathbb{R}_{+},\mathrm{Pr}(\mathbb{R}))$, the Cauchy--Stieltjes transform of $\nu_{t}$ is defined as%define
\begin{equation*}
\psi_{\nu}(t,z):=\int_{\mathbb{R}}f_{z}(x)\nu_{t}(dx).
\end{equation*}
We identify the weak limit of the sequence $\{(\langle
	\mu_{t}^{(n)},f\rangle)_{t\geq0}:n\geq1\}$ to be the family $%
	(\mu_{t})_{t\geq0}$ such that its Cauchy--Stieltjes transform satisfies the Burgers equation that characterizes the Cauchy--Stieltjes transform of the law of the
	free L\'{e}vy process $\left \{ Z_{t}:t\geq0\right \} ,$ appearing in the
	proof of Theorem 5.10 in Bercovici and Voiculescu \cite{BV} (see also \cite{ca} proof of Theorem III.2 and \cite{HS} (Theorem 4.5 and Appendix A.1))). This is shown in the following result.

\begin{theorem}
\label{thfbm} Assume that $\mu_{0}^{(n)}$ converges weakly to $\delta_{0}$.
Then the family of measure-valued processes $\{(\mu_{t}^{(n)})_{t\geq0}:n%
\geq1\}$ converges weakly in $\mathcal{D}(\mathbb{R}_{+},\mathrm{Pr}(\mathbb{%
R}))$ to a unique continuous probability-measure valued function $%
(\mu_{t})_{t\geq0},$ satisfying for each $t\geq0$,
\begin{equation}\label{limit_eq}
{\frac{\partial}{\partial t}\psi_{\mu}(t,z)}=-\sigma^{2}\psi_{\mu}(t,z)\frac{%
\partial}{\partial z}\psi_{\mu}(t,z)-\eta \frac{\partial}{\partial z}%
\psi_{\mu}(t,z)-\frac{\partial}{\partial z}\psi_{\mu}(t,z)\int _{\mathbb{R}%
\backslash \{0\} }\frac{\psi_{\mu}(t,z)+r}{1-r\psi_{\mu}(t,z)}\rho(dr).
\end{equation}
\end{theorem}

The following auxiliary lemma can be obtained from the proof of Lemma III.6 in \cite{ca}.
\begin{lemma}\label{adap cabanal}
Let $X=(X_1,\dots,X_n)$, $Y=(Y_1,\dots,Y_n)$ be random vectors and let $V=(V_1,\dots,V_n)$ be an independent random vector Haar distributed on $\mathbb{S}(\mathbb{C}_{n})$. Then for any   $\varepsilon>0$ and any bounded function $f:\mathbb{R}\mapsto\mathbb{C}$ there exists a constant $C(f)>0$ not dependent on $X$ such that
\begin{align*}
	\mathbb{P}\left(\Bigg|\sum_{i=1}^n |V_i|^2f(X_i)-\frac{1}{n}\sum_{i=1}^n f(X_i)\Bigg|\geq\varepsilon\right)\leq
	C(f)\exp \left( -\frac{(n-1)}{8\Vert f \Vert_{\infty}^{ 2}%
	}\varepsilon^{2}\right).
\end{align*}
\end{lemma}

The following lemma will be useful for identifying the limit law
of the sequence of empirical measures $\{(\mu^{(n)})_{t\geq0}:n\geq1\}$. 
\begin{lemma}\label{converg}
For $z \in \mathbb{C}\backslash \mathbb{R}$ and $t>0$,
\begin{align*}
\lim_{n\to\infty}\mathbb{E}\Bigg[\Bigg|&\int_{0}^{t}\int_{\mathbb{S}(\mathbb{C}_{n})}\int
_{\mathbb{R}}\mathrm{tr}\Big [f_z(X_{u-}^{(n)}+rvv^{\ast})-f_z(X_{u-}^{(n)})-\mathcal{D}f_z(X_{u-}^{(n)})(rvv^*)1_{\{|r|\leq 1\}}\Big]\notag\\
&\hspace{1cm}-\mathrm{tr}\Big [f_z\left(X_{u-}^{(n)}+\frac{r}{n}\mathrm{I_n}\right)-f_z(X_{u-}^{(n)})-\mathcal{D}f_z(X_{u-}^{(n)})\left(\frac{r}{n}\mathrm{I_n}\right)1_{\{|r|\leq 1\}}\Big]\rho_{n}^{\alpha}(dr)\pi(dv)du\Bigg|\Bigg]=0.
\end{align*}
\end{lemma}
\begin{proof}
 	(i) First we will obtain some estimations needed for the rest of the proof. For $r\in \mathbb{R}$, and $\xi>0$ let us denote by $\{\lambda_i^{(n)}(u,v,r,\xi)\}_{i=1}^n$ to the family of eigenvalues of the matrix  $X_{u-}^{(n)}+\frac{r}{n}\mathrm{I_n}+r\xi\left(vv^{\ast}-\frac{1}{n}\mathrm{I_n}\right)$. Additionally, we  define the diagonal matrix $\Lambda(u,v,r,\xi):=\sum_{i=1}^n\lambda_i^{(n)}(u,v,r,\xi)P_i$, where $P_{i}$ are the projections onto the coordinate axes of $\mathbb{H}_n$.

We note that by using Taylor's theorem we obtain for any $r\in \mathbb{R}$
\begin{align*}
&\int_{\mathbb{S}(\mathbb{C}_{n})}\Bigg|\mathrm{tr}\Big [f_z(X_{u-}^{(n)}+rvv^{\ast})-f_z\left(X_{u-}^{(n)}+\frac{r}{n}\mathrm{I_n}\right)\Big]\Bigg|\pi(dv)\notag\\
&\leq\int_{\mathbb{S}(\mathbb{C}_{n})}\int_0^1\Bigg|r\mathrm{tr}\Bigg [\mathcal{D}f_z\left(X_{u-}^{(n)}+\frac{r}{n}\mathrm{I_n}+r\xi\left(vv^{\ast}-\frac{1}{n}\mathrm{I_n}\right)\right)\left(vv^{\ast}-\frac{1}{n}\mathrm{I_n}\right)\Bigg]\Bigg|d\xi\pi(dv)\notag\\
&=\int_{\mathbb{S}(\mathbb{C}_{n})}\int_0^1\Bigg|r\mathrm{tr}\left[\mathcal{D}f_z(\Lambda(u,v,r,\xi))\circ\left(vv^{\ast}-\frac{1}{n}\mathrm{I_n}\right)\right]\Bigg|d\xi\pi(dv)\\
&=\int_{\mathbb{S}(\mathbb{C}_{n})}\int_0^1|r|\Bigg|\sum_{i=1}^n |v_i|^2f_z'(\lambda_i^{(n)}(u,v,r,\xi))-\frac{1}{n}\sum_{i=1}^n f_z'(\lambda_i^{(n)}(u,v,r,\xi))\Bigg|d\xi\pi(dv),
\end{align*}
where in order to obtain the first equality we applied \eqref{iden_fd} together with the fact that under $\pi$, $vv^*$ is invariant under unitary conjugations. \newline
Therefore, for every $\varepsilon>0$
\begin{align*}
&\mathbb{E}\Bigg[\int_{\mathbb{S}(\mathbb{C}_{n})}\Bigg|\mathrm{tr}\Big [f_z(X_{u-}^{(n)}+rvv^{\ast})-f_z\left(X_{u-}^{(n)}+\frac{r}{n}\mathrm{I_n}\right)\Big]\Bigg|\pi(dv)\Bigg]\notag\\
&\leq\int_{\mathbb{S}(\mathbb{C}_{n})}\int_0^12|r|\|f_z\|_{\infty}\mathbb{P}\Bigg(\Bigg|\sum_{i=1}^n |v_i|^2f_z'(\lambda_i^{(n)}(u,v,r,\xi))-\frac{1}{n}\sum_{i=1}^n f_z'(\lambda_i^{(n)}(u,v,r,\xi))\Bigg|\geq\varepsilon\Bigg)d\xi\pi(dv)\notag\\
&+\varepsilon\int_{\mathbb{S}(\mathbb{C}_{n})}\int_0^1|r|d\xi\pi(dv).
\end{align*}
Hence, using Lemma \ref{adap cabanal}
we obtain that for every $r\in \mathbb{R}$, and $\varepsilon>0$
\begin{align}\label{nv_00}
\mathbb{E}\Bigg[\int_{\mathbb{S}(\mathbb{C}_{n})}\Bigg|\mathrm{tr}\Big [f_z(X_{u-}^{(n)}+rvv^{\ast})-f_z&\left(X_{u-}^{(n)}+\frac{r}{n}\mathrm{I_n}\right)\Big]\Bigg|\pi(dv)\Bigg]\notag\\
&\leq |r|\left[\varepsilon+2\|f_z'\|_{\infty}C(f_z)\exp \left( -\frac{(n-1)}{8\Vert f_z \Vert_{\infty}^{ 2}%
}\varepsilon^{2}\right)\right].
\end{align}

Finally, using \eqref{iden_fd} together with Lemma \ref{adap cabanal}
\begin{align}\label{nv_4a}
\mathbb{E}\Bigg[\int_{\mathbb{S}(\mathbb{C}_{n})}\Bigg|\mathrm{tr}\Big [\mathcal{D}f_z(X_{u-}^{(n)})&\left(vv^*-\frac{1}{n}\mathrm{I_n}\right)\Big]\Bigg|\pi(dv)\Bigg]\leq \mathbb{E}\left[\int_{\mathbb{S}(\mathbb{C}_{n})}\Bigg|\sum_{i=1}^nf^{\prime}_z(\lambda_i^{(n)}(X_{u-}^{(n)}))\left(|v_i|^2-\frac{1}{n}\right)\Bigg|\pi(dv)\right]\notag\\&\leq \varepsilon +2\|f'_z\|_{\infty}\int_{\mathbb{S}(\mathbb{C}_{n})}\mathbb{P}\left(\Bigg|\sum_{i=1}^nf^{\prime}_z(\lambda_i^{(n)}(X_{u-}^{(n)}))\left(|v_i|^2-\frac{1}{n}\right)\Bigg|\geq \varepsilon\right)\pi(dv)\notag\\
&\leq \varepsilon +2\|f'_z\|_{\infty}C(f_z)\exp \left( -\frac{(n-1)}{8\Vert f_z \Vert_{\infty}^{ 2}%
}\varepsilon^{2}\right).
\end{align}

(ii) Let us consider $\beta\in(1/4,1/2)$. Then using \eqref{nv_00} with $\varepsilon=n^{-\beta}$ and the fact that $\alpha\in(0,1/4)$ we obtain
\begin{align}\label{nv_2}
&\mathbb{E}\Bigg[\Bigg|\int_{0}^{t}\int_{\mathbb{S}(\mathbb{C}_{n})}\int
_{|r|\geq 1}\mathrm{tr}\Big [f_z(X_{u-}^{(n)}+rvv^{\ast})-f_z(X_{u-}^{(n)})\Big]\rho_{n}^{\alpha}(dr)\pi(dv)du\notag\\
&\hspace{2cm}-\int_{0}^{t}\int_{\mathbb{S}(\mathbb{C}_{n})}\int
_{|r|\geq 1}\mathrm{tr}\Big [f_z\left(X_{u-}^{(n)}+\frac{r}{n}\mathrm{I_n}\right)-f_z(X_{u-}^{(n)})\Big]\rho_{n}^{\alpha}(dr)\pi(dv)du\Bigg|\Bigg]\notag\\
&\leq t\int_{|r|\geq 1}|r|\left(\frac{1+r^2}{r^2}\right)1_{(1/n,n^{2\alpha}/(n^\alpha-1))}(|r|)\rho(dr)\left[\frac{1}{n^\beta}+2\|f_z'\|_{\infty}C(f_z)\exp \left( -\frac{(n-1)}{8\Vert f_z \Vert_{\infty}^{ 2}%
}\frac{1}{n^{2\beta}}\right)\right]\notag\\
&\leq t\int_{|r|\geq 1}\left(\frac{1+r^2}{r^2}\right)\rho(dr)\left[\frac{n^{2\alpha}}{n^\beta(n^{\alpha}-1)}+2\|f_z'\|_{\infty}C(f_z)\frac{n^{2\alpha}}{(n^{\alpha}-1)}\exp \left( -\frac{(n-1)}{8\Vert f_z \Vert_{\infty}^{ 2}%
}\frac{1}{n^{2\beta}}\right)\right]\overset{n\uparrow\infty}{\longrightarrow}0.
\end{align}

(iii) Using \eqref{nv_00} together with \eqref{nv_4a} we obtain for each $r\in \mathbb{R}$
\begin{align}\label{nv_6a}
&\limsup_{n\to\infty}\mathbb{E}\Bigg[\int_{\mathbb{S}(\mathbb{C}_{n})}\Bigg|\mathrm{tr}\Big [f_z(X_{u-}^{(n)}+rvv^{\ast})-f_z(X_{u-}^{(n)})-\mathcal{D}f_z(X_{u-}^{(n)})(rvv^*)\Big]\notag\\
&\hspace{2cm}-\mathrm{tr}\Big [f_z\left(X_{u-}^{(n)}+\frac{r}{n}\mathrm{I_n}\right)-f_z(X_{u-}^{(n)})-\mathcal{D}f_z(X_{u-}^{(n)})\left(\frac{r}{n}\mathrm{I_n}\right)\Big]\Bigg|\pi(dv)\Bigg]\notag\\
&\leq \limsup_{n\to\infty}\Bigg\{\mathbb{E}\Bigg[\int_{\mathbb{S}(\mathbb{C}_{n})}\Bigg|\mathrm{tr}\Big [f_z(X_{u-}^{(n)}+rvv^{\ast})-f_z\left(X_{u-}^{(n)}+\frac{r}{n}\mathrm{I_n}\right)\Big]\Bigg|\pi(dv)\Bigg]\notag\\&+\mathbb{E}\Bigg[\int_{\mathbb{S}(\mathbb{C}_{n})}\Bigg|\mathrm{tr}\Big [\mathcal{D}f_z(X_{u-}^{(n)})\left(rvv^*-\frac{r}{n}\mathrm{I_n}\right)\Big]\Bigg|\pi(dv)\Bigg]\Bigg\}\notag\\
&\leq \limsup_{n\to\infty}2|r|\left[\varepsilon+2\|f_z'\|_{\infty}C(f_z)\exp \left( -\frac{(n-1)}{8\Vert f_z \Vert_{\infty}^{ 2}%
}\varepsilon^{2}\right)\right]\leq 2|r|\varepsilon.
\end{align}
So taking $\varepsilon\downarrow0$ in \eqref{nv_6a} we obtain
\begin{align}\label{nv_5a}
&\lim_{n\to\infty}\mathbb{E}\Bigg[\int_{\mathbb{S}(\mathbb{C}_{n})}\Bigg|\mathrm{tr}\Big [f_z(X_{u-}^{(n)}+rvv^{\ast})-f_z(X_{u-}^{(n)})-\mathcal{D}f_z(X_{u-}^{(n)})(rvv^*)\Big]\notag\\
&\hspace{2cm}-\mathrm{tr}\Bigg [f_z\left(X_{u-}^{(n)}+\frac{r}{n}\mathrm{I_n}\right)-f_z(X_{u-}^{(n)})-\mathcal{D}f_z(X_{u-}^{(n)})\left(\frac{r}{n}\mathrm{I_n}\right)\Bigg]\Bigg|\pi(dv)\Bigg]=0.
\end{align}
We also note that by using Taylor's theorem and Lemma \ref{der_bound}(ii)
\begin{align}\label{nv_4}
	&\int_{\mathbb{S}(\mathbb{C}_{n})}\Bigg|\mathrm{tr}\Big [f_z(X_{u-}^{(n)}+rvv^{\ast})-f_z(X_{u-}^{(n)})-\mathcal{D}f_z(X_{u-}^{(n)})(rvv^*)\Big]\pi(dv)\notag\\
	&\hspace{3cm}-\mathrm{tr}\Big [f_z\left(X_{u-}^{(n)}+\frac{r}{n}\mathrm{I_n}\right)-f_z(X_{u-}^{(n)})-\mathcal{D}f_z(X_{u-}^{(n)})\left(\frac{r}{n}\mathrm{I_n}\right)\Big]\Bigg|\pi(dv)\notag\\
	&\leq \int_{\mathbb{S}(\mathbb{C}_{n})}\int_0^1r^2(1-\xi)\Bigg|\mathrm{tr}\Bigg [\mathcal{D}^2f_z(X_{u-}^{(n)}+r\xi vv^*)\left((vv^{\ast})^2\right)-\mathcal{D}^2f_z\left(X_{u-}^{(n)}+\frac{r\xi}{n} \mathrm{I_n}\right)\left(\left(\frac{1}{n}\mathrm{I_n}\right)^2\right)\Bigg]\Bigg|d\xi\pi(dv)\notag\\
	&\hspace{2cm}\leq4\|f_z''\|_{\infty}r^2\int_0^1(1-\xi)d\xi=2r^2\|f_z''\|_{\infty}.
\end{align}
Hence by \eqref{nv_4} we can apply dominated convergence and using \eqref{nv_5a} we obtain
\begin{align}\label{nv_7}
\lim_{n\to\infty}\mathbb{E}\Bigg[\Bigg|&\int_{0}^{t}\int_{\mathbb{S}(\mathbb{C}_{n})}\int
_{|r|\leq 1}\mathrm{tr}\Big [f_z(X_{u-}^{(n)}+rvv^{\ast})-f_z(X_{u-}^{(n)})-\mathcal{D}f_z(X_{u-}^{(n)})(rvv^*)\Big]\notag\\
&\hspace{2cm}-\mathrm{tr}\Big [f_z\left(X_{u-}^{(n)}+\frac{r}{n}\mathrm{I_n}\right)-f_z(X_{u-}^{(n)})-\mathcal{D}f_z(X_{u-}^{(n)})\left(\frac{r}{n}\mathrm{I_n}\right)\Big]\rho_{n}^{\alpha}(dr)\pi(dv)du\Bigg|\Bigg]\notag\\
&\leq \int_{0}^{t}\int
_{|r|\leq 1}\lim_{n\to\infty}\mathbb{E}\Bigg[\int_{\mathbb{S}(\mathbb{C}_{n})}\Bigg|\mathrm{tr}\Big [f_z(X_{u-}^{(n)}+rvv^{\ast})-f_z(X_{u-}^{(n)})-\mathcal{D}f_z(X_{u-}^{(n)})(rvv^*)\Big]\notag\\
&\hspace{0.6cm}-\mathrm{tr}\Bigg [f_z\left(X_{u-}^{(n)}+\frac{r}{n}\mathrm{I_n}\right)-f_z(X_{u-}^{(n)})-\mathcal{D}f_z(X_{u-}^{(n)})\left(\frac{r}{n}\mathrm{I_n}\right)\Bigg]\Bigg|\pi(dv)\Bigg]\left(\frac{1+r^2}{r^2}\right)\rho(dr)du=0.
\end{align}

The result now follows from \eqref{nv_2} and \eqref{nv_7}.
\end{proof}

\begin{proof}[Proof of Theorem \protect\ref{thfbm}]

Following the discussion previous to Theorem \protect\ref{thfbm}, it is enough to prove the convergence to \eqref{limit_eq} for a suitable subsequence.
From Proposition \ref{tightness}, the family $\{(\mu_{t}^{(n)})_{t\geq0}:n%
\geq1\}$ is relatively compact. Hence, there exists a subsequence
$\left \{ n_{k} \right \} _{k\geq1}$ such that $\{(\mu_{t}^{(n_{k})})_{t\geq0}:k\geq1\}$
converges weakly to some $(\mu_{t})_{t\geq0}$ in $\mathbb{D}(\mathbb{R}_+,\text{Pr}(\mathbb{R}))$.

First we note that using Lemma III.7 in \cite{ca} together with \eqref{aux_n1_3}
\begin{align}\label{nv_8}
\int_{0}^{t}\int_{\mathbb{S}(\mathbb{C}_{n_k})}\int
	_{\mathbb{R}}\mathrm{tr}\Big [f_z\left(X_{u-}^{(n_k)}+\frac{r}{n_k}\mathrm{I_{n_k}}\right)&-f_z(X_{u-}^{(n_k)})-\mathcal{D}f_z(X_{u-}^{(n_k)})\left(\frac{r}{n_k}\mathrm{I_{n_k}}\right)1_{\{|r|\leq 1\}}\Big]\rho_{n_k}^{\alpha}(dr)\pi(dv)du\notag\\
&=\int_{0}^{t}\int _{\mathbb{R}}\left[\frac{\displaystyle -r\frac{\partial}{\partial z}\psi_{\mu^{(n_k)}}(s,z)}{\displaystyle 1-r\psi_{\mu^{(n_k)}}(s,z)}+r\frac{\partial}{\partial z}\psi_{\mu^{(n_k)}}(s,z)1_{\{|r|\leq1\}}\right]\rho_{n_k}^{\alpha}(dr)ds.
\end{align}
Next, by (\ref{SDEa}) together with \eqref{nv_8} we obtain
\begin{align}\label{cst}
	& \psi_{\mu^{(n_k)}}(t,z)-\psi_{\mu^{(n_k)}}(0,z)+{\sigma_{n_k}}^{2}\int_{0}^{t}\psi_{\mu^{(n_k)}}(s,z)\frac{%
		\partial}{\partial z}\psi_{\mu^{(n_k)}}(s,z)ds+{\gamma}\int_{0}^{t}\frac{\partial}{%
		\partial z}\psi_{\mu^{(n_k)}}(s,z)ds\notag\\
	&-\int_{0}^{t}\int _{\mathbb{R}}\left[\frac{\displaystyle -r\frac{\partial}{\partial z}\psi_{\mu^{(n_k)}}(s,z)}{\displaystyle 1-r\psi_{\mu^{(n_k)}}(s,z)}+r\frac{\partial}{\partial z}\psi_{\mu^{(n_k)}}(s,z)1_{\{|r|\leq1\}}\right]{\rho}_{n_k}^{\alpha}(dr)ds=\Psi_{n_k}(t,z),
\end{align}	
where
\begin{align*}
	&\Psi_{n_k}(t,z)=M_{t}^{n_k,f_z}+\int_{0}^{t}\int_{\mathbb{S}(\mathbb{C}_{n_k})}\int
	_{\mathbb{R}}\Bigg\{\mathrm{tr}\Big [f_z(X_{u-}^{(n_k)}+rvv^{\ast})-f_z(X_{u-}^{(n_k)})-\mathcal{D}f_z(X_{u-}^{(n_k)})(rvv^*)\Big]\notag\\
	&\hspace{3.5cm}-\mathrm{tr}\Bigg [f_z\left(X_{u-}^{(n_k)}+\frac{r}{n_k}\mathrm{I_{n_k}}\right)-f_z(X_{u-}^{(n_k)})-\mathcal{D}f_z(X_{u-}^{(n_k)})\left(\frac{r}{n_k}\mathrm{I_{n_k}}\right)\Bigg]\Bigg\}\rho_{n_k}^{\alpha}(dr)\pi(dv)du.
\end{align*}
By Lemmas \ref{martl} and \ref{converg} we obtain
\begin{align}\label{limit}
\lim_{k\to\infty}\Psi_{n_k}(t,z)=0,\qquad\text{in probability.}
\end{align}
Hence from \eqref{cst}, \eqref{limit} and the continuous mapping theorem we have that the Cauchy--Stieltjes transform of any weak limit $(\mu_t)_{t\geq0}$ of the subsequence $\{(\mu_{t}^{(n_{k})})_{t\geq0}:k\geq1\}$ should satisfy the following partial differential equation
\begin{align}
\frac{\partial}{\partial t}&\psi_{\mu}(t,z)  =-{\sigma^{2}}\psi_{\mu }(t,z)%
\frac{\partial}{\partial z}\psi_{\mu}(t,z)-{\gamma}\frac{\partial }{\partial
z}\psi_{\mu}(t,z)-\frac{\partial}{\partial z}\psi_{\mu}(t,z)\int_{0<|r|\leq1}%
\frac{\psi_{\mu} (t,z)}{1-r\psi_{\mu}(t,z)}(1+r^2)\rho(dr)  \notag
\\
& -\frac{\partial}{\partial z}\psi_{\mu}(s,z)\int_{|r|>1}\frac{r}{%
1-r\psi_{\mu} (s,z)}\left(\frac{1+r^2}{r^2}\right)\rho(dr)  \notag \\
& =-{\sigma^{2}}\psi_{\mu}(t,z)\frac{\partial}{\partial z}\psi_{\mu
}(t,z)-\eta \frac{\partial}{\partial z}\psi_{\mu}(t,z)-\frac{\partial}{%
\partial z}\psi_{\mu}(s,z)\left( \int_{0<|r|\leq1}\frac{\psi_{\mu}(t,z)}{%
1-r\psi_{\mu }(t,z)}(1+r^2)\rho(dr)+\int_{|r|\leq1}r\rho(dr)\right)
\notag \\
& -\frac{\partial}{\partial z}\psi_{\mu}(s,z)\left( \int_{|r|>1}\frac {r}{%
1-r\psi_{\mu}(s,z)}\left(\frac{1+r^2}{r^2}\right)\rho(dr)-\int_{|r|>1}\frac{1}{r}%
\rho(dr)\right)  \notag \\
& =-{\sigma^{2}}\psi_{\mu}(t,z)\frac{\partial}{\partial z}\psi_{\mu
}(t,z)-\eta \frac{\partial}{\partial z}\psi_{\mu}(t,z)-\frac{\partial}{%
\partial z}\psi_{\mu} (t,z)\int_{\mathbb{R}\backslash \{0\}}\frac{%
\psi_{\mu}(t,z)+r}{1-r\psi_{\mu}(t,z)}\rho(dr).  \notag
\end{align}
\end{proof}

\end{document}